\documentclass[11pt]{amsart}

\usepackage[utf8x]{inputenc}
\usepackage{amsmath}
\usepackage{amssymb}
\usepackage{amsthm}
\usepackage[foot]{amsaddr}
\usepackage[toc,page]{appendix}
\usepackage{enumerate}
\usepackage{fullpage}
\usepackage{graphicx}
\usepackage{lmodern}
\usepackage{mathtools}
\usepackage{multirow}
\usepackage[numbers]{natbib}
\usepackage{remreset}
\usepackage{tabularx}
\usepackage{subcaption}
\usepackage[usenames,dvipsnames]{color}
\usepackage[draft]{changes}
\usepackage{diagbox}

\theoremstyle{plain}
\newtheorem{theorem}{Theorem}[section]

\newtheorem{lemma}{Lemma}[section]
\numberwithin{equation}{section}

\newtheorem{assumption}{Assumption}[section]

\theoremstyle{definition}

\newtheorem{rem}{Remark}

\newcommand{\algns}[1]{\begin{align*} #1 \end{align*}}
\newcommand{\algnd}[1]{\begin{aligned} #1 \end{aligned}}
\newcommand{\eqn}[1]{\begin{equation} #1 \end{equation}}
\newcommand{\eqns}[1]{\begin{equation*} #1 \end{equation*}}

\newcommand{\ass}[2]{\begin{equation} \tag{A.#1}#2 \end{equation}}

\newcommand{\setof}[1]{\left\{ #1 \right\} }	% To define sets
\newcommand{\sothat}{\,:\,}
\newcommand{\dual}[1]{#1^{'}}
\newcommand{\adjoint}[1]{#1^*}

\newcommand{\mat}[1]{\mathsf{#1}}

\newcommand{\T}{\mathcal{T}}
\newcommand{\N}{\mathcal{N}}
\newcommand{\inv}[1]{#1^{-1}}

\newcommand{\grad}{\nabla}

\newcommand{\intd}{\mathrm{d}}	% For integration

\newcommand{\pd}{\partial}
\newcommand{\norm}[1]{\left\Vert#1\right\Vert}
\newcommand{\norw}[2]{\left\Vert#1\right\Vert_{#2}}
\newcommand{\inner}[2]{\left( #1, #2\right)}

\newcommand{\reals}{\mathbb{R}}		% Real numbers
\begin{document}
	\title[]{Multigrid Methods for Discrete Fractional Sobolev Spaces}\thanks{Thanks}
	\author{Trygve B\ae rland$^\dagger$}
	\email{trygveba@math.uio.no}
	\address{$^\dagger$Department of Mathematics, University of Oslo, Blindern, Oslo, 0316 Norway}
	
	\author{Miroslav Kuchta$^\dagger$}
	\email{mirok@math.uio.no}
	
	\author{Kent-Andre Mardal$^\dagger$}
	\email{kent-and@math.uio.no}

	\begin{abstract} 
		Coupled multiphysics problems often give rise to interface conditions naturally formulated in fractional Sobolev spaces. Here, both positive- and negative fractionality are common. When designing efficient solvers for discretizations of such problems it would then be useful to have a preconditioner for the fractional Laplacian.
		In this work, we develop an additive multigrid preconditioner for the fractional Laplacian with positive fractionality, and show a uniform bound on the condition number. 
		For the case of negative fractionality, we re-use the preconditioner developed for the positive fractionality and left-right multiply a regular Laplacian with a preconditioner with positive fractionality to obtain the desired
negative fractionality. Implementational issues are outlined in details as the differences between the discrete operators and their corresponding matrices must be addressed when realizing these algorithms in code.  We finish with some numerical experiments verifying the theoretical findings.
	\end{abstract}
	\maketitle
	
	\section{Introduction}
        \label{sec:intro}
	
	Multiphysics or multiscale problems often involve coupling conditions
	at interfaces which are manifolds of lower dimensions. The coupling 
	conditions are, because of the lower dimensionality,  naturally
	posed in fractional Sobolev spaces, and this fact seemingly complicates   
	discretization schemes and solution algorithms. Our focus here will be on the
	development of solution algorithms in terms of multilevel preconditioners that
	from an implementational point of view only require minor adjustments of 
	standard multilevel algorithms.  
	
	As simplified 
	examples of problems involving interface conditions,  let us consider the following
	two prototype problems. First an elliptic problem with a trace constraint 
	\eqn{
		\label{frac:model1}
		\algnd{
			-\Delta u + T^*\lambda &=& f, \quad x \in \Omega, \\  
			T u &=& g, \quad x \in \Gamma ,  
		}
	}  
	and second an elliptic problem in mixed form with a trace constraint
	\eqn{
		\label{frac:model2}
		\algnd{
			u - \nabla p + T^*\lambda &=& f, \quad x \in \Omega, \\  
			\nabla \cdot u &=& g, \quad x \in \Omega, \\ 
			T u &=& h, \quad x \in \Gamma .
		}
	}
	Here, $\Gamma$ is a sub-manifold either within $\Omega$ or at its boundary, 
	$T$ is a trace operator and $T^*$ its adjoint. Both 
	problems are assumed to be equipped with suitable boundary conditions. 
	We remark that although these problems are single physics problems, they may easily be coupled to other
	problems through the Lagrange multiplier at the interface. As such, the problems represent well the
	challenge of handling the interface properly in a multiphysics setting.  
	
	%The first example is an elliptic equation and a weak formulations leads to seeking a solution in $u\in H^1(\Omega)$. Therefore,   
	%$Tu \in H^{1/2}(\Gamma)$ and the duality with the Lagrange multiplier $\lambda \in H^{-1/2}$. For the second problem, with a mixed formulation
	%of an elliptic equation, $u \in H(\mbox{div})$, $T u \in  H^{-1/2}$ 
	%and $\lambda\in H^{1/2}$.
	
	We may write the above problems as 
	\[  
	\left(
	\begin{array}{cc}
	-\Delta & T^*  \\
	T & 0 \end{array} 
	\right)
	\left(
	\begin{array}{c}
	u  \\
	\lambda 
	\end{array} 
	\right)
	= 
	\left(
	\begin{array}{c}
	f  \\
	g 
	\end{array} 
	\right)
	\quad\mbox{ and }\quad  
	\left(
	\begin{array}{ccc}
	I & -\nabla & T^*  \\
	\nabla\cdot & 0 & 0  \\
	T & 0  & 0 \end{array} 
	\right)
	\left(
	\begin{array}{c}
	u  \\
	p  \\
	\lambda 
	\end{array} 
	\right)
	= 
	\left(
	\begin{array}{c}
	f  \\ g  \\ h 
	\end{array} 
	\right) . 
	\]
	A crucial challenge is to discretize and solve these problems in a scalable way such that the computations scales linearly with 
	the number of unknowns. Our approach here is to consider iterative methods 
	and develop preconditioners that are both spectrally equivalent with the involved operators 
	and of order-optimal complexity. The main difficulty is the handling of the Lagrange multiplier which 
	falls outside the scope of  standard multilevel methods. To provide a general framework, we will consider preconditioners constructed in 
	terms of the so-called  operator preconditioning approach~\cite{mardal2011preconditioning} to be used 
	for iterative methods. As will be explained later, the block diagonal preconditioners constructed by 
	this technique will be of the following form:   
	\[  
	\left(
	\begin{array}{cc}
	-\Delta^{-1} & 0  \\
	0 & (-\Delta)^{-1/2} \end{array} 
	\right)
	\quad\mbox{ and }\quad  
	\left(
	\begin{array}{ccc}
	(I - \nabla\nabla\cdot)^{-1} & 0 & 0  \\
	0 & I & 0  \\
	0 & 0  & (-\Delta)^{1/2} \end{array} 
	\right),
	\]
	respectively.
	Multilevel methods spectrally equivalent with both $(-\Delta)^{-1}$ and $(I - \nabla\nabla\cdot)^{-1}$ are well known.  
	The challenging part in both cases is the construction of efficient preconditioning algorithms 
	that approximate the inverse of the fractional Laplace problems on 
	the form 
	\begin{equation}
	\label{frac:lap}
	(-\Delta)^s u = f, \quad x \in \Gamma
	\end{equation}
	with $s=1/2$ and $s=-1/2$, equipped with suitable boundary conditions. 
	Furthermore, if $\Gamma$ is of codimension 2, numerical simulations~\cite{kuchta2016preconditioning}
	indicates that $s\in (-0.2, -0.1)$ gives rise to efficient preconditioners. In this 
	paper we therefore consider methods for $s \in [-1,1]$.

	There are many examples of applications of fractional Laplacians in the litterature
	and we mention a few that motivates this work. 
	For non-overlapping domain decomposition
	preconditioner are studied in \cite{arioli2012discrete}, \cite{kocvara2016constraint}. Here, they use
	\eqref{frac:lap} with $s=\tfrac{1}{2}$ to precondition the interface problem
	involving the related Steklov-Poincar{\' e} operator. In \cite{kuchta2016efficient}
	the authors use \eqref{frac:lap} with $s=-\tfrac{1}{2}$ as part of a block
	diagonal preconditioner for multiphysics
	problem where the constraint coupling two domains of different topological
	dimension is enforced by the Lagrange multiplier. Therein the fractionality $s$
	is dictated by the mapping properties of the Schur complement operator. Some
	further examples of coupled systems with domains of different dimensionality
	include Babu{\v s}ka's problem for enforcing Dirichlet boundary
	conditions on an elliptic operator \cite{babuvska1973finite}, 
	flow stabilization by removal of tangential velocity at the boundary 
	through Lagrange multipliers~\cite{BERTOLUZZA201758}, 
	the no-slip condition on the surface of a falling solid in the Navier-Stokes fluid~\cite{court2015fictitious},
	inextensibility constraint in the complex model of vesicle formation~\cite{aland2014diffuse}, and
	the potential jump on a membrane of a cardiac cell~\cite{tveito2017cell}. We note that in these
	applications the fractional Laplace problem has to be solved with both positive
	and negative exponent.
	
	% How to compute
	%For $s > 0$ there exists a number of solution techniques for \eqref{frac:lap}.
	%In the matrix transfer technique of \cite{ilic2005numerical}, \cite{ilic2006numerical}
	%the approximate solution can be obtained using a matrix representation of
	%powers of the discrete Laplacian. Here, the methods based
	%on exact diagonalization are, in general, resctricted to small matrices or
	%stuctured domains, e.g. on tensor product domains the matrix representation
	%with respect to the Fourer basis is trivial to diagonalize, e.g. \cite{bueno2014fourier}, while
	%\cite{peisker1988numerical} exploits the structure of finite difference Laplacian to evaluate
	%the powers using the fast Fourier transform. In case of general domains/discretizations
	%scalable methods for evaluating the powers are based on approximate action
	%of the matrix function $f(y) = y^{-s}$.
	
	% Foo approx

	There are several alternative approaches that have been used in order to approximate fractional Laplacians. 
	Polynomial approximations of discretizations of $A^s$, where $A$ is a discrete Laplacian, can be computed with the standard Krylov
	subspace methods. However, without any preconditioner a Krylov subspace of large
	dimension is required for convergence, see e.g. Lanczos method in \cite[Section 4]{knizhnerman2010new}.  
	We note that if \eqref{frac:lap} is used to form a preconditioner the approximation
	of $A^s$  can be less accurate, and in \cite{arioli2012discrete} (generalized) Lanczos method
	is used to construct an efficient preconditioner. Therein the efficiency is
	due to a small size of the subspace. In \cite{yang2011novel} Lanczos method with a
	preconditioner based on the invariant subspace of the $k$ smallest eigenvalues
	is proposed for solving the fractional heat equation. The method is shown to
	generate non-polynomial approximations of $A^s$. The contour integral method
	of \cite{hale2008computing} and the extended Krylov method of \cite{knizhnerman2010new} are then related to
	rational function approximations of  $A^s$, while \cite{harizanov2016optimal} consider the best
	uniform rational aproximations of the trasformed function $A \mapsto A^{\beta-s}$.
	In general, the approximation properties of these methods depend on the condition number
	of $A$ and
	thus computations of extremal eigenvalues are often part of the algorithm. Further,
	computational complexity of the methods depends on efficient solvers for auxiliary
	linear systems, e.g. $(A-q_kI)x=b$ in \cite{harizanov2016optimal} where $q_k\in\reals$ is a
	shift parameter. Almost mesh independent preconditioners for systems arising in
	\cite{hale2008computing} and \cite{knizhnerman2010new} are discussed in \cite{burrage2012efficient}.
	An alternative approach to the matrix transfer method is presented in \cite{bonito2015numerical}
	where the inverse of the fractional Laplacian is defined via the (integral) Balakrishnan
	formula \cite{balakrishnan1960fractional}. 

	Multilevel methods have been considered for fractional Laplacians have been considered
	in \cite{bramble2000computational, oswald1998multilevel, harizanov2016optimal}, but there
	seems to be a significant untapped potential for advancement. Our
	work here is closely related to  \cite{bramble2000computational}, where order-optimal preconditioners for $A^s$ when $s\in\left(-\frac{3}{2}, \frac{3}{2} \right)$ were constructed 
	using an hierarchical basis approach. The paper did, however, only consider smoothers based
	on level-dependent scaling and did not put much focus on the actual implementation. 
	%\MK{$\leftarrow$I can't make sense of the end of the sentence. limited?}
	Here, we will develop and analyze a multilevel algorithm that is straightforward
	to implement in a standard multilevel software framework. In fact, 
	the main change required is an adjustment of the smoothers. To illustrate
	the change, let us assume that we want to solve the system $\mat{A}\mat{x} = \mat{b}$, where
	$\mat{A}$ is a stiffness matrix corresponding to a discretized Laplacian.     
	A standard Jacobi algorithm can then be written 
	\[
	\mat{x}^{n+1}_i = \mat{x}^{n}_{i} - \frac{1}{\mat{A}_{i,i}}\mat{r}_i^n,  
	\]
	where $\mat{A}_{i,i}$ are the diagonal entries of the stiffness matrix for a 
	discretized Laplacian, and $\mat{r}^n$ is the residual of the $n$'th iterate, $\mat{x}^n$ . 
	In our case, for $\mat{A}^s\mat{x} = \mat{b}$,
	the proposed Jacobi smoother 
	may be implemented as
	\[
	\mat{x}^{n+1}_i = \mat{x}^n_{i} - \left(\frac{1}{\mat{M}_{i,i}^{1-s}\mat{A}_{i,i}^s}\right) \mat{r}^n_i,  .  
	\]
	Here, $\mat{M}_{i,i}$ are the diagonal entries of the mass matrix. 
	We notice here that for $s=0$ the action is a Jacobi iteration on
	the mass matrix, for $s=1$ the action is a Jacobi iteration on 
	the stiffness matrix, and for $0<s<1$ the action is an interpolation
	between these two extremes. 
	From an implementational point of view, the restriction and interpolation operators used are the same as those used in
	standard multilevel algorithms. However, from a theoretical point of view, 
	the fact that we use standard restriction and interpolation operators, means
	that the multilevel approach will be non-nested.  In fact, the matrices on 
	coarser levels do not correspond to $(-\Delta)^s$-Galerkin projections 
	of the matrix on the finer levels. We therefore employ the framework
	of non-nested multilevel methods~\cite{bramble1991analysis}. Furthermore, 
	a multiplicative multilevel algorithm would require computing the residual and hence the
	evaluation of the exact $(-\Delta)^{-s}$ operator on every level. Since the 
	evaluation of the exact $(-\Delta)^{-s}$ is a computationally expensive procedure, 
	we instead rely on the additive multilevel algorithm proposed in \cite{bramble1990parallel}, where
	the same residual is used on all levels. The additive variant is significantly
	less efficient than corresponding multiplicative variants in terms of the 
	conditioning in the sense that the conditioning depends on the number of levels. 
	Still, this is a small price to pay (only logarithmic in the number of unknowns) 
	to avoid exact evaluation of the residual.  
	In this paper we will assume quasi-uniform mesh and continuous piecewise linear finite elements. This is mainly for simplicity, and the results can be generalized to higher order discretizations, as well as discontinuous Galerkin methods.
	%\KAM{assumptions here, quasi-uniform er vel kanskje vanskelig å fjerne. }
	
	The paper is structured as follows. In Section \ref{sec:preliminaries}, we introduce notation, and some useful operator inequalities related to fractional powers of positive operators. We also give a brief 
	discussion of fractional Sobolev spaces, as well as of some implementational concerns. Section \ref{sec:abstract-bpx} is devoted to the analysis of an abstract multilevel framework. In section \ref{sec:discrete-bpx} we use this framework to define operators that are spectrally equivalent to the inverse of the fractional Laplacian when the fractionality $s \geq 0$. We discuss some strategies for preconditioning when $s<0$ in section \ref{sec:preconditioning-negative-s}, and in section \ref{sec:implementation} we discuss implementation of the preconditioners developed in the previous sections.  Finally, we provide numerical results that verify our theoretical result in section \ref{sec:numerical-experiments}.
	
	\section{Notation and preliminaries}
	\label{sec:preliminaries}
	% FILE: preliminaries.tex
	%!TeX root=fractional_mg_main.tex
	
	Let $\Omega$ be a bounded, Lipschitz domain in $\reals^n$, with boundary $\pd \Omega$. We denote by $L^2(\Omega)$ the space of square integrable functions over $\Omega$, with inner product  $\inner{\cdot}{\cdot}$ and norm $\norm{\cdot}$. For  $k \in \mathbb{N}$, we denote by $H^k(\Omega)$ the usual Sobolev spaces of functions in $L^2(\Omega)$ with all derivatives up to order $k$ in $L^2(\Omega)$. 
	The norm and inner product in $H^k$ is denoted by $\|\cdot\|_k$ and $(\cdot, \cdot)_k$, respectively. 
	The closure in $H^k$ of smooth functions with compact support in $\Omega$ is denoted as $H^k_0(\Omega)$ and its dual
	space is $H^{-k}$. In general a Hilbert space $X$ is equipped with a norm $\|\cdot\|_X$ and an inner product $(\cdot, \cdot)_X$ 
	and the dual space is denoted $\dual{X}$.  
	
	Let $A$ be a symmetric positive-definite operator on a finite-dimensional Hilbert space $X$ with dimension $N$. Denote by $\setof{(\lambda_k,\phi_k)}_{k=1}^N$ the set of eigenpairs of $A$, normalized so that
	\eqns{
		\inner{\phi_k}{\phi_l}_X = \delta_{k,l},
	}
	where $\delta_{k,l}$ is the Kronecker delta.  Then $\phi_k$, for $k=1,\ldots,N$ forms an orthonormal basis of $X$, and if $u \in X$ has the representation $u = \sum_{k=1}^N c_k \phi_k$, then
	
	\eqns{
		Au = \sum_{k=1}^N \lambda_k c_k \phi_k.
	}
	For $s\in \reals$, we define the fractional power $A^s$ of $A$ by
	\eqns{
		A^s u = \sum_{k=1}^N \lambda_k^s c_k \phi_k.
	}
	If $A$ is only positive semi-definite, then we must restrict to $s \geq 0$, and the eigenvectors corresponding to the nullspace of $A$ are left out (also for $s=0$).  
	If $B$ is another symmetric positive semi-definite operator on $X$, we write $A \leq B$ if for every $u \in X$
	\eqns{
		\inner{Au}{u}_X \leq \inner{Bu}{u}_X
	}
	holds. Note that $A \geq 0$ is equivalent to saying that $A$ is positive semi-definite. 
	
	A result in operator theory is the Löwner-Heinz inequality, which in our case states that if $A \leq B$, then
	\eqn{
		\label{eq:Loewner-Heinz}
		A^s \leq B^s, \quad s\in [0,1],
	}
	cf. for instance \cite{kato1952notes}. Inequality \eqref{eq:Loewner-Heinz} means that the function $x^s$ with $x\in [0,\infty)$ is operator monotone for $s\in[0,1]$.
	It follows that  $-(x)^s$ is operator convex (cf. \cite{hansen1982monotone}), that is, for any two symmetric positive semi-definite operators $A$ and $B$ on a Hilbert space $X$, the inequality
	\eqns{
		\label{eq:operator-convex-def}
		\lambda A^s + (1-\lambda)B^s \leq \left(\lambda A + (1-\lambda)B \right)^s
	}
	holds for every $\lambda \in [0,1]$.
	A key result regarding operator convex functions is the Jensen's operator inequality (cf. \cite[Theorem 2.1]{hansen2003jensen}). The version we will use in the current work states that for any $K\in \mathbb{N}$ and $s\in [0,1]$
	\eqn{
		\label{eq:jensen-inequality}
		\sum_{k=1}^K \adjoint{P_k}A_k^s P_k \leq \left(\sum_{k=1}^K \adjoint{P_k}A_k P_k\right)^s,
	}
	where for $k=1,\ldots,K$, $A_k$ are symmetric positive semi-definite operators on $X$, and $P_k$ are linear operators on $X$ so that $\sum_{k=1}^K \adjoint{P}_k P_k \leq 1$.

% INTERPOLATION THEORY:
	\subsection{Fractional Sobolev spaces}
	\label{sec:interpolation_theory}
	We consider the interpolation spaces between $H^1(\Omega)$ and $L^2(\Omega)$ as defined in \cite{lions1972nonhom}. Let the inner product on $H^1(\Omega)$ be realized by the operator $A := I - \Delta $, as
	\eqns{
		\inner{u}{v}_1 = \inner{Au}{v} = \inner{u}{v} + \inner{\grad u}{\grad v}, \quad u,v \in H^1(\Omega).
	}
	$A$ is unbounded as an operator mapping $L^2(\Omega)$ to $L^2(\Omega)$.  However, $A$ is well-defined on the set
	\eqns{
		D(A) = \setof{u \in L^2(\Omega) \sothat Au \in L^2(\Omega)},
	}
	which is a dense subspace of $L^2(\Omega)$. On $D(A)$, $A$ is symmetric and positive-definite, and so the fractional powers of $A$,  $A^\theta$ for $\theta\in \reals$, are well-defined. Note that in the particular case $\theta = \frac{1}{2}$,
	\eqns{
		\norm{A^{\frac{1}{2}}u}^2 = \inner{Au}{u} = \norm{u}_1.
	}
	For $s\in[0,1]$, we define the fractional Sobolev spaces as
	\eqn{
		\label{eq:Frac-Sobolev-space-def}
		H^s(\Omega) = \setof{u \in L^2(\Omega) \sothat A^{\frac{s}{2}}u \in L^2(\Omega) },
	}
	which is a Hilbert space with inner product given by
	\eqns{
		\inner{u}{v}_s = \inner{A^s u}{v}, \quad u,v\in H^s(\Omega),
	}
	and we denote the corresponding norm by $\norm{\cdot}_s$.
	
	We define $H_0^s(\Omega)$ to be the closure of $C_0^\infty(\Omega)$, the space of infinitely smooth functions with compact support in $\Omega$, in the norm of $H^s(\Omega)$. We note that if $s \leq \frac{1}{2}$, the spaces $H_0^s(\Omega)$ and $H^s(\Omega)$ coincide (cf. \cite[Theorem 11.1]{lions1972nonhom}). For $s \in [-1,0]$, we define a family of fractional Sobolev spaces using the dual of $H_0^s(\Omega)$. That is,
	\eqns{
		H^s(\Omega) = \dual{\left(H^{-s}_0(\Omega)\right)}.
	}
	
	Replacing $H^1(\Omega)$ with $H_0^1(\Omega)$ and setting $A = -\Delta$ in the above construction, will again yield the space $H^s_0(\Omega)$, with equivalent norm, for all $s$ except when $s = \frac{1}{2}$. In this case, interpolation between $H_0^1(\Omega)$ and $L^2(\Omega)$ results in a space that is strictly contained in $H_0^{\frac{1}{2}}(\Omega)$. The subsequent analysis is valid for both $H^s_0(\Omega)$ and $H^s(\Omega)$.  
	
	We remark that the above defined fractional space $H^s(\Omega)$ is equivalent to 
	the fractional space $\hat{H}^s(\Omega)$ defined in terms of the norm
	\[ 
	\norm{u}^2_{\hat{H}^{s}(\Omega)} =  \norm{u}^2 +  
	\int_{\Omega \times \Omega} \frac{|u(x) -
		u(y)|}{|x-y|^{n+s}}\, \intd x \intd y.
	\]
	A detailed overview of the various definitions of fractional Sobolev norms and 
	their discretizations can be found in \cite{lischke2018fractional}.  
	
	\subsection{Discrete fractional Sobolev spaces}
	
	We will now consider a discretization of the fractional Sobolev spaces $H_0^s(\Omega)$ and $H^{-s}(\Omega)$ for $s\in[0,1]$.
	Let $X_h$ be a finite-dimensional subspace of $H_0^1(\Omega)$, with $\dim X_h = N_h$. We define the operator $A_h: X_h \to X_h$ by
	\eqn{
		\label{eq:Ah-def}
		\inner{A_h u}{v} = \inner{\grad u}{\grad v}, \quad u,v \in X_h.
	}
	Using the fractional powers of $A_h$, we define for $s\in \reals$ the discrete fractional inner product on $X_h$ by
	\eqns{
		\inner{u}{v}_{s,h} = \inner{A^s_h u}{v}, \quad u,v \in X_h,
	}
	and denoted the corresponding norm  by $\norm{\cdot}_{s,h}$. 
	It is clear that for $s=0$ and $s=1$, the two norms $\norm{\cdot}_{s,h}$ and $\norm{\cdot}_s$ coincide on $X_h$. Therefore,  due to \cite[Lemma 2.3]{arioli2009discrete}, the norms $\norw{\cdot}{s,h}$ and $\norw{\cdot}{s}$, when $s\in [0,1]$, are equivalent on $X_h$, with constants of equivalence independent of $N_h$.
	
	Let $X_H$ be a subspace of $X_h$, and $A_H: X_H \to X_H$ be defined analogously to  $A_h$ in \eqref{eq:Ah-def}. If $I_H: X_H \to X_h$ is the inclusion map, we see that
	\eqn{
		\label{eq:MhAh-inheritance}
		A_H = \adjoint{I}_HA_hI_H,
	}
	where $\adjoint{I}_H$ is the adjoint of $I_H$ with respect to the $L^2$ inner product.
	
	We may also define $A^s_H: X_H \to X_H$, but generally, $A^s_H \not= \adjoint{I}_H A^s_h I_H$.  However, by Jensen's operator inequality we have the following.
	\begin{lemma}
		\label{lem:subspace-As-estimate}
		For every $s \in [0,1]$ we have
		\eqns{
			\label{eq:subspace-As-estimate-operator}
			\adjoint{I_H}A^s_h I_H \leq A^s_H.
		}
		That is, for every $u \in X_H$,
		\eqn{
			\label{eq:subspace-As-estimate}
			\inner{A^s_h u}{ u} \leq \inner{A^s_H u}{u}.
		}
	\end{lemma}
	\begin{proof}
		For $s=0$ and $s=1$, \eqref{eq:subspace-As-estimate} holds with equality, so let $0<s<1$. We start by noticing that since $\adjoint{I_H}I_H$ is the identity on $X_H$,
		\eqns{
			A^2_H = (\adjoint{I_H}I_HA_H \adjoint{I_H}I_H)^2 = \adjoint{I_H}(I_H A_H \adjoint{I_H})^2 I_H.
		}
		By induction, we find that
		\eqns{
			A^k_H = \adjoint{I_H}(I_H A_H \adjoint{I_H})^k I_H
		}
		for every nonnegative integer $k$. It follows that for any polynomial $p$
		\eqn{
			\label{eq:subspace-As-estimate-proof1}
			p(A_H) = \adjoint{I_H}\,p\left( I_H A_H \adjoint{I_H} \right) I_H.
		}
		
		Take now $\epsilon > 0$. The spectra of both $A_H$ and $I_H A_H \adjoint{I_H}$ are contained in some bounded, nonnegative interval $[0,b]$. By Weierstrass' approximation Theorem we can thus choose a polynomial $p$ so that
		\eqns{
			\norm{\left(I_H A_H \adjoint{I_H} \right)^s - p(I_H A_H \adjoint{I_H} )} < \epsilon, \quad \text{ and } \quad \norm{A_H^s - p(A_H)} < \epsilon.
		}
		Using the triangle inequality and \eqref{eq:subspace-As-estimate-proof1} we have that
		\eqns{
			\norm{A_H^s -\adjoint{I_H}\left( I_H A_H \adjoint{I_H} \right)^s I_H} \leq \norm{A_H^s - p(A_H)} + \norm{\adjoint{I_H}\left( \left(I_H A_H \adjoint{I_H} \right)^s - p(I_H A_H \adjoint{I_H} )\right) I_H} < 2\epsilon,
		}
		and since $\epsilon$ was arbitrary, this shows that
		\eqn{
			\label{eq:subspace-As-estimate-proof2}
			A^s_H =	\adjoint{I_H}(I_H A_H \adjoint{I_H})^s I_H.
		}
		
		Using \eqref{eq:MhAh-inheritance} in \eqref{eq:subspace-As-estimate-proof2}, we get that
		\eqn{
			\label{eq:subspace-As-estimate-proof3}
			A^s_H =	\adjoint{I_H}(I_H \adjoint{I_H}A_h I_H \adjoint{I_H})^s I_H.
		} 
		Finally, $I_H \adjoint{I_H}$ defines a symmetric operator on $X_h$ with $L^2$ operator norm equal to $1$. Since the function $x \mapsto -x^s$ is operator convex on $[0,\infty)$ we can use Jensen's operator inequality \eqref{eq:jensen-inequality} in \eqref{eq:subspace-As-estimate-proof3} to get
		\algns{
			A^s_H &\geq \adjoint{I_H}I_H \adjoint{I_H} A^s_h I_H \adjoint{I_H} I_H \\
			& = \adjoint{I_H} A^s_h I_H,
		}
		where we have used that $\adjoint{I_H}I_H$ is the identity on $X_H$. 
	\end{proof}

% ABSTRACT MULTILEVEL THEORY:
	\section{Abstract multilevel theory}
	\label{sec:abstract-bpx}
	In order to analyze and implement a multigrid preconditioner for the fractional Laplacian there are three main issues
	that need to be dealt with. First, we need to  derive and implement a smoother with the desired properties. 
	As already mentioned in the introduction, this step only requires a minor modification to standard smoothing 
	algorithms. We will discuss the details concerning implementation later. 
	Second, the restriction/interpolation operators do not 
	result in a nested hierarchy of operators in our fractional setting as 
	$A^s_H \not= I^*_H A_h I_H$. For this reason we will employ the framework for non-nested
	multilevel algorithms developed in ~\cite{bramble1991analysis}.
	Third, our main motivation for developing fractional multilevel solvers is their application to 
	multiphysics and multiscale problems where the preconditioner for the fractional Laplacian 
	is utilized at the interfaces. As such, the fractional Laplacian operator is not 
	part of the original problem and we may therefore not assume that this operator 
	has been implemented. Furthermore, implementing this operator in an efficient 
	manner is a challenge, as the exact fractional operator involves
	the solution of a large, global and expensive generalized eigenvalue problem.  
	To avoid the application of the fractional Laplacian on the various levels
	we employ additive multilevel schemes which enable the residual of the problem to be used at all levels and remove 
	the need for implementing a global fractional Laplacian operator. That said, 
	the theory developed here extends to multiplicative algorithms for problems involving the fractional Laplacian, such as the standard V-cycle.   
	In this section we will address the second and the third issues and outline a theory for an additive multilevel scheme applied to an  abstract 
	non-nested problem. As such analysis of this section is the synthesis of the papers~\cite{bramble1990parallel} and~\cite{bramble1991analysis}.    
	
	%A common assumption in multilevel algorithms for discretizations of partial differential equations is to assume that the discrete operators on each level are inherited from the discrete operator on the above level.  As we have seen the previous section, this is in general not true for fractional powers of differential operators. However, the inequality in Lemma \ref{lem:subspace-As-estimate} coincides with one of the key assumptions used in the abstract multigrid framework of \cite{bramble1991analysis}.
	%
	%Another problem is when calculating residuals in a standard multigrid algorithm, we need to compute the action of $A^s_h$. This, we have seen, reduces to solving a generalized eigenvalue problem, which can be prohibitively expensive. To avoid this problem, we resort to an additive multigrid method (cf. e.g. \cite{bramble1990parallel}). In an additive multilevel method, the smoothing step and subspace correction are done on the initial input on all levels, and there is no need for computing any residuals. Therefore, the following presentation will be similar to \cite{bramble1990parallel}, but with a generalization to noninherited bilinear forms, akin to the work done in \cite{bramble1991analysis}.
	
	Assume that we are given a nested sequence of finite-dimensional function spaces
	\eqns{
		V_1 \subset V_2 \subset \cdots \subset V_J = V, \quad J \geq 2.
	}
	We further assume that $V$, and consequently all subspaces of $V$, is endowed with an inner product $\inner{\cdot}{\cdot}$, with corresponding induced norm $\norm{\cdot}$. Moreover, for each $k=1,\ldots, J$, we assume we are given a symmetric positive definite operator $A_k: V_k \to V_k$, and we set $A = A_J$. Note that we do not assume that the $A_k$ operators are nested.
	
	For the development and analysis of a multilevel algorithm, it will be useful to define a number of operators on each level $k$. First, we define $P_{k,k-1}: V_k \to V_{k-1}$ by
	\eqn{
		\inner{A_{k-1}P_{k,k-1}v}{w} = \inner{A_k v}{w}, \quad \forall v \in V_k, w \in V_{k-1}.
	}
	We remark that in a nested setting, $P_{k,k-1}$ is the $A$-projection, while since the $A_k$ operators are not nested, the $P_{k,k-1}$ operators are not projections.
	Next, we define $Q_k: V \to V_k$ by
	\eqn{
		\inner{Q_k v}{w} = \inner{v}{w}, \quad \forall v \in V, w \in V_k.
	}
	It follows by the above definitions that
	\eqn{
		\label{eq:AP_QA}
		A_{k-1}P_{k,k-1} = Q_{k-1}A_k,
	}
	and $Q_l Q_k = Q_k Q_l = Q_l$ whenever $l \leq k$.
	For the sake of brevity, it will also be useful to define $P_k: V \to V_k$ by $P_k = P_{k+1,k} P_{k+2,k+1} \cdots P_{J,J-1}$. Using the definition of $P_{j+1,j}$ for $j=k,\ldots,J-1$ we see that
	\eqns{
		\inner{A_kP_kv}{w} = \inner{Av}{w}, \quad \forall v \in V, w\in V_k.
	}
	Furthermore, applying \eqref{eq:AP_QA} repeatedly, we find that 
	\eqn{
		\label{eq:AP_QA-all-levels}
		A_k P_k = Q_k A.
	}
	
	Finally, suppose we are given for each $k$ a smoother, which is a symmetric positive definite operator $R_k : V_k \to V_k$, which in some sense should approximate $\inv{A_k}$
	on $V_k \backslash V_{k-1}$.
	We can now define a additive multilevel operator $B: V \to V$ by
	\eqn{
		\label{eq:BPX-def}
		B = \sum_{k=1}^J R_k Q_k.
	}
	As remarked in \cite{bramble1990parallel}, $B$ can viewed as a additive version of the standard multiplicative V-cycle multigrid algorithm, where $R_k$ plays the role of smoother. Because of this, it is reasonable that the assumptions we need to make to establish spectral equivalence between $\inv{A}$ and $B$ are similar to those made for standard multigrid algorithms.
	
	We assume that for $k=2,\ldots,J$
	\ass{1}{
		\label{ass:noninheritance}
		\inner{A_k v}{v} \leq \inner{A_{k-1}v}{v}, \quad \forall v \in V_{k-1}.
	}
	Under assumption \eqref{ass:noninheritance} and the definition of $P_{k,k-1}$ we see that for any $v \in V_k$
	\algns{
		\inner{A_{k-1}P_{k,k-1}v}{P_{k,k-1}v} &= \inner{A_k v}{P_{k,k-1}v} \\
		&\leq \inner{A_k P_{k,k-1}v}{P_{k,k-1}v}^{\frac{1}{2}}\inner{A_k v}{v}^{\frac{1}{2}} \\
		& \leq \inner{A_{k-1} P_{k,k-1}v}{P_{k,k-1}v}^{\frac{1}{2}}\inner{A_k v}{v}^{\frac{1}{2}}.
	}
	Thus, we see that \eqref{ass:noninheritance} implies 
	\eqn{
		\label{eq:noninheritance-adjoint}
		\inner{A_{k-1}P_{k,k-1}v}{P_{k,k-1}v} \leq \inner{A_k v}{v}, \quad \forall v \in V_k.
	}
	Conversely, assume \eqref{eq:noninheritance-adjoint}. Then, for any $v \in V_{k-1}$ and the definition of $P_{k,k-1}$
	\algns{
		\inner{A_k v}{v} &= \inner{A_{k-1}P_{k,k-1}v}{v} \\
		&\leq \inner{A_{k-1}P_{k,k-1}v}{P_{k,k-1}v}^{\frac{1}{2}} \inner{A_{k-1}v}{v}^{\frac{1}{2}} \\
		&\leq \inner{A_{k}v}{v}^{\frac{1}{2}} \inner{A_{k-1}v}{v}^{\frac{1}{2}},
	}
	which implies \eqref{ass:noninheritance}. Thus, \eqref{ass:noninheritance} and \eqref{eq:noninheritance-adjoint} are equivalent.
	Notice that a similar inequality to \eqref{eq:noninheritance-adjoint} would also hold for $P_k$, namely
	\eqn{
		\label{eq:noninheritance-adjoint-all-levels}
		\inner{A_k P_k v}{P_k v} \leq \inner{A v}{v}, \quad \forall v \in V.
	}
	
	For the operators $R_k$, we assume there are constants $C_1, C_2 > 0$, independent of $k$ so that
	\ass{2}{
		\label{ass:smoother-stability}
		C_1 \frac{\norm{v}^2}{\lambda_k} \leq \inner{R_k v}{v} \leq C_2 \inner{\inv{A_k}v}{v}, \quad \forall v \in V_k,
	}
	where $\lambda_k$ is the largest eigenvalue of $A_k$.  Lastly, as is common in multigrid theory, we will use an approximation assumption to establish spectral equivalence between $B$ and $\inv{A}$. In this work, we assume the following approximation property: That there is an $\alpha \in (0,1]$ and constant $C_3 > 0$, independent of $k$, so that
	\ass{3}{
		\label{ass:approximation-property}
		\inner{A_k(I-P_{k,k-1})v}{v} \leq C_3^\alpha \left( \frac{\norm{A_k v}^2}{\lambda_k} \right)^\alpha \inner{A_k v}{v}^{1-\alpha}, \quad \forall v \in V_k.
	} 
	
	We are now in a position to state and prove the main theorem of this section:
	\begin{theorem}
		\label{thm:bpx-spectral-equivalence}
		Assume that \eqref{ass:noninheritance}, \eqref{ass:smoother-stability}, and \eqref{ass:approximation-property} hold. Then, with $B$  given in \eqref{eq:BPX-def},
		\eqn{
			\label{eq:bpx-spectral-equivalence}
			C_1 \inv{C_3}J^{1-\frac{1}{\alpha}}\inner{A v}{v} \leq \inner{B Av}{Av} \leq C_2J\inner{A v}{v}.
		}
		holds for every $v \in V$.
	\end{theorem}
	\begin{proof}
		Fix $v \in V$.
		Using the definition of $B$ together with \eqref{eq:AP_QA-all-levels} we find that
		\eqns{
			\inner{BA v}{Av} = \sum_{k=1}^J \inner{R_k Q_k A v}{Q_k A v} = \sum_{k=1}^J \inner{R_k A_k P_k v}{A_k P_k}.
		}
		Thus, applying the second inequality of \eqref{ass:smoother-stability} and \eqref{eq:noninheritance-adjoint-all-levels} gives
		\eqns{
			\inner{B Av}{Av} \leq C_2\sum_{k=1}^J \inner{A_k P_kv}{P_k v} \leq C_2 J \inner{A v}{v},
		}
		which proves the second inequality of \eqref{eq:bpx-spectral-equivalence}.
		
		For the first inequality of \eqref{eq:bpx-spectral-equivalence} we write
		\eqns{
			v = \sum_{k=1}^J (P_k - P_{k-1})v,
		}
		where we interpret $P_0 = 0$ and $P_J =I$. By the definition of $P_k$, we have that $P_{k-1} = P_{k,k-1}P_k$, and so
		\eqns{
			v = \sum_{k=1}^J (I - P_{k,k-1})P_k v.
		}
		It follows that
		\eqns{
			\inner{A v}{v} = \sum_{k=1}^J \inner{A_k(I-P_{k,k-1})P_k v}{P_k v}.
		}
		Using \eqref{ass:approximation-property} and \eqref{eq:noninheritance-adjoint-all-levels}, gives
		\algns{
			\inner{Av}{v} &\leq C_3^\alpha \sum_{k=1}^J\left(\inv{\lambda_k}\norm{A_kP_kv}^2\right)^\alpha \inner{A_k P_k v}{P_k v}^{1-\alpha} \\
			&\leq C_3^\alpha \inner{A v}{v}^{1-\alpha}\sum_{k=1}^J\left(\inv{\lambda_k}\norm{A_kP_kv}^2\right)^\alpha.
		}
		The first inequality of \eqref{ass:smoother-stability} then implies that
		\algns{
			\inner{A v}{v} &\leq (\inv{C_1}C_3)^\alpha \inner{A v}{v}^{1-\alpha}\sum_{k=1}^J\inner{R_kA_k P_k v}{A_k P_k v}^\alpha \\
			&\leq (\inv{C_1}C_3)^\alpha J^{1-\alpha} \inner{A v}{v}^{1-\alpha}\left(\sum_{k=1}^J\inner{R_kA_k P_k v}{A_k P_k v}\right)^\alpha \\
			&\leq (\inv{C_1}C_3)^\alpha J^{1-\alpha} \inner{A v}{v}^{1-\alpha}\inner{B A v}{A v}^\alpha,
		}
		where the second step follows by Hölder's inequality. The last step follows by the definition of \eqref{eq:BPX-def} and \eqref{eq:AP_QA-all-levels}. Dividing by $(\inv{C_1}C_3)^\alpha \inner{A v}{v}^{1-\alpha}J^{1-\alpha}$ on both sides and raising to the power $\frac{1}{\alpha}$ gives the first inequality of \eqref{eq:bpx-spectral-equivalence}. 
	\end{proof}
	
	\begin{rem}
		\label{rem:bpx-weak-spectral-equivalence}
		Analagously to what was done in \cite{bramble1990parallel}, we can replace assumption \eqref{ass:approximation-property} with an assumption on the projections $Q_k$. In particular, if instead of \eqref{ass:approximation-property}, we assume that there is a constant $C_4 > 0$, independent of $k$ so that
		\eqns{
			\norm{(I-Q_{k-1})v}^2 \leq C_4 \inv{\lambda_k} \inner{A v}{v}, \quad \forall v \in V,
		}
		then we can use an argument like what was made in \cite[Theorem 1 and Corollary 1]{bramble1990parallel} to show that
		\eqn{
			\label{eq:bpx-weak-spectral-equivalence}
			\inv{C_4}C_1\inv{J}\inner{A v}{v} \leq \inner{B Av}{Av} \leq C_2 J \inner{A v}{v}
		}
		for every $v \in V$.
	\end{rem}
	
% APPLICATION TO DISCRETE SOBOLEV NORMS:
	\section{Preconditioner for discrete fractional Laplacian}
	\label{sec:discrete-bpx}
	In this section we use the abstract theory developed in Section \ref{sec:abstract-bpx} to derive an order optimal preconditioner for the discrete fractional Laplacian $A^s_h$, described in Section \ref{sec:preliminaries}, when $s \in [0,1]$.
	
	Let $\Omega$ be a bounded, polygonal domain in $\reals^n$ and suppose we are given a quasi-uniform triangulation of $\Omega$, denoted by $\T_h$, where $h$ denotes the characteristic mesh size. We restrict our discussion to the case when $V_h$ is the space of continuous, piecewise linear functions relative to the triangulation $\T_h$ which vanish on $\pd \Omega$. To define a nested sequence of subspaces, we suppose that $\T_h$ is constructed by successive refinements. That is, we are given a sequence,
	\eqns{
		\T_1 \subset \cdots \T_J =\T_h,
	}
	of quasi-uniform triangulations, and $\T_k$ has characteristic mesh size $h_k$ for $k=1,\ldots,J$. In the following, we will assume the bounded refinement hypothesis, that is, $h_{k-1} \leq \gamma h_k$ for $k=2,\ldots,J$, where $\gamma \geq 1$ is a constant. In practice, $\gamma$ is around $2$.
	For each $k$ we define $V_k$ as the space of continuous, piecewise linear functions relative to $\T_k$ that vanish on $\pd \Omega$. Further, we define $A_k: V_k \to V_k$ by
	\eqns{
		\inner{A_k v}{w} = \inner{\grad v}{\grad w}, \quad v,w\in V_k.
	}
	
	We now fix $s \in [0,1]$. Since $A_k$ is symmetric positive definite, we can define SPD operators $A^s_k$ and corresponding norms
	\eqns{
		\norm{v}_{s,k}^2 := \inner{A^s_k v}{v}, \quad v\in V_k.
	}
	Note that if $s=0$ or $s=1$, the norm $\norm{\cdot}_{s,k}$ coincides with the $L^2$- and $H^1_0$-norm, respectively. That is, $\norm{\cdot}_{0,k} = \norm{\cdot}$, and $\norm{\cdot}_{1,k} = \norm{\cdot}_{1}$.
	
	Analogous to the discussion in Section \ref{sec:abstract-bpx} we also define operators $P^s_{k,k-1}: V_k \to V_{k-1}$ by
	\eqn{
		\label{eq:Ps_k-definition}
		\inner{A^s_{k-1}P^s_{k,k-1}v}{w} = \inner{A^s_k v}{w}, \quad \forall v \in V_k, w \in V_{k-1}.
	}
	$Q_{k}: V_J \to V_k$ as the $L^2$-projection, and $P^s_k := P^s_{k+1,k} P^s_{k+2,k+1}\cdots P^s_{J,J-1}$. 
	
	To complete the description of a multilevel preconditioner, we still need to define smoothers, $R^s_k$, for each $k$ and $s$. In this work, we will define additive smoothers based on domain decomposition. To that end, let $\N_k$ be the set of vertices in the triangulation $\T_k$, and for each $\nu \in \N_k$, let $\T_{k,\nu}$ be the set of triangles meeting at the vertex $\nu$. Then $\T_{k,\nu}$ forms a triangulation of a small subdomain $\Omega_{k,\nu}$. We define $V_{k,\nu}$ to be the subspace of functions in $V_k$ with support contained in $\bar{\Omega}_{k,\nu}$.  Analogously to what we did for $V_k$, we may define for each $\nu \in \N_k$ operators $A^s_{k,\nu}: V_{k,\nu} \to V_{k,\nu}$, and $Q_{k,\nu}: V_k \to V_{k,\nu}$. For $k=2,\ldots,J$, we define
	\eqn{
		\label{eq:Rk-definition}
		R^s_k := \sum_{\nu \in \N_k} A^{-s}_{k,\nu}Q_{k,\nu},
	}
	while on the coarsest level we set $R^s_1 = A^{-s}_1$. We note that the smoothers are symmetric positive-definite, and their inverse satisfy
	\eqn{
		\label{eq:smoother-inverse-norm}
		\inner{\inv{(R^s_k)}v}{v} = \inf_{\underset{v_\nu \in V_{k,\nu}}{v = \sum_{\nu} v_\nu}} \sum_{\nu \in \N_k} \inner{A^s_{k,\nu}v_\nu}{v_\nu}, \quad v \in V_k.
	}
	Our preconditioner now reads
	\eqn{
		\label{eq:BPXs-preconditioner-definition}
		B_h^s := \sum_{k=1}^J R^s_k Q_k.
	}
	
	We want to apply Theorem \ref{thm:bpx-spectral-equivalence} to the preconditioner defined by \eqref{eq:BPXs-preconditioner-definition} and \eqref{eq:Rk-definition}, so we need to verify assumptions \eqref{ass:noninheritance}-\eqref{ass:approximation-property}.
	
	Using Lemma \ref{lem:subspace-As-estimate}, we immediately find that for every $k$,
	\eqns{
		\inner{A^s_k v}{v} \leq \inner{A^s_{k-1} v}{v}, \quad \forall v \in V_{k-1}, 
	}
	which verifies \eqref{ass:noninheritance} in the current context.
	
	We present the verification of \eqref{ass:smoother-stability} in the following Lemma.
	\begin{lemma}
		\label{lem:smoother-stability-verification}
		For $k=1,\ldots,J$, let $R^s_k: V_k \to V_k$ be defined by \eqref{eq:Rk-definition}. Assume that there exists a constant $K_0$, so that for every $k$ and $v \in V_k$ there exists a decomposition $\sum_{\nu \in \N_k}v_\nu = v$, with $v_\nu \in V_{k,\nu}$ so that
		\eqn{
			\label{eq:L2-stable-decomposition}
			\sum_{\nu \in \N_k}\norm{v_\nu}^2 \leq K_0 \norm{v}^2.
		}
		Then there are constants $C_1, C_2 > 0$, so that for every $k$, 
		\eqn{
			\label{eq:smoother-stability-BPXs}
			C_1 \frac{\norm{v}^2}{\lambda^s_k} \leq \inner{R^s_k v}{v} \leq C_2 \inner{A^{-s}_k v}{v}, \quad \forall v \in V_k,
		}
		where $\lambda_k^s$ is the largest eigenvalue of $A^s_k$.
	\end{lemma}
	\begin{proof}
		It is evident that \eqref{eq:smoother-stability-BPXs} holds at the coarsest level, i.e.,  for $k=1$ \eqref{eq:smoother-stability-BPXs} is satisfied with $C_1 = C_2 = 1$. So let $k \geq 2$, and fix $v \in V_k$.
		For $\nu \in \N_k$, let $\lambda_{k,\nu}^s$ denote the largest eigenvalue of $A^s_{k,\nu}$.
		To prove the first inequality in \eqref{eq:smoother-stability-BPXs}, we begin by noting that for every $\nu \in \N_k$
		\eqns{
			\lambda_k^1 = \sup_{w \in V_k}\frac{\inner{A^1_k w}{w}}{\inner{w}{w}} \geq \sup_{w \in V_{k,\nu}}\frac{\inner{A^1_{k,\nu} w}{w}}{\inner{w}{w}} = \lambda_{k,\nu}^1.
		}
		Thus, since $\lambda_k^s = (\lambda_k^1)^s$, we have that 
		\eqn{
			\label{eq:smoother-stability-proof1}
			\lambda_k^s \geq \lambda_{k,\nu}^s.
		}
		Now, using \eqref{eq:smoother-stability-proof1} and assumption \eqref{eq:L2-stable-decomposition}, together with the definition of $Q_{k,\nu}$ and $R^s_k$ yields
		\algns{
			\frac{\inner{v}{v}}{\lambda_k^s} &= \frac{1}{\lambda_k^s}\sum_{\nu \in \N_k}\inner{v}{v_\nu} \\
			&= \frac{1}{\lambda_k^s}\sum_{\nu \in \N_k}\inner{Q_{k,\nu}v}{v_\nu} \\
			&\leq \left(\frac{1}{\lambda_k^s}\sum_{\nu \in \N_k}\inner{Q_{k,\nu}v}{Q_{k,\nu}v}\right)^{\frac{1}{2}} \left(\frac{1}{\lambda_k^s}\sum_{\nu \in \N_k}\norm{v_\nu}^2\right)^{\frac{1}{2}} \\
			&\leq  \left(\sum_{\nu \in \N_k}\frac{1}{\lambda_{k,\nu}^s}\inner{Q_{k,\nu}v}{Q_{k,\nu}v}\right)^{\frac{1}{2}}\left(\frac{K_0}{\lambda_k^s}\norm{v}^2\right)^{\frac{1}{2}} \\
			&\leq \left(\sum_{\nu \in \N_k}\inner{A^{-s}_{k,\nu}Q_{k,\nu}v}{Q_{k,\nu}v}\right)^{\frac{1}{2}}\left(\frac{K_0}{\lambda_k^s}\norm{v}^2\right)^{\frac{1}{2}} \\
			&\leq \inner{R^s_k v}{v}^{\frac{1}{2}} \left(\frac{K_0}{\lambda_k^s}\norm{v}^2\right)^{\frac{1}{2}},
		}
		which proves the first inequality of \eqref{eq:smoother-stability-BPXs} with $C_1 = \inv{K_0}$.
		
		For the second inequality, we begin by noting that for $s=1$, it was proven in \cite[Lemma 7.2]{xu1992iterative} that
		there is a constant $C$, independent of $k$ so that
		\eqns{
			\inner{R^1_k v}{v} \leq C\inner{\inv{A_k}v}{v}, \quad \forall v \in V_k.
		}
		Since $s\in [0,1]$, it follows by the Löwner-Heinz inequality \eqref{eq:Loewner-Heinz} that
		\eqn{
			\label{eq:smoother-stability-proof3}
			\inner{(R^1_k)^s v}{v} \leq C^s\inner{A^{-s}_k v}{v}, \quad \forall v \in V_k.
		}
		Thus, if we can show that 
		\eqn{
			\label{eq:smoother-stability-proof4}
			\inner{R^s_k v}{v} \leq C \inner{(R^1_k)^s v}{v},
		}
		for some constant $C$, which is independent of $k$, then we can use \eqref{eq:smoother-stability-proof3} together with \eqref{eq:smoother-stability-proof4} to prove the second inequality of \eqref{eq:smoother-stability-BPXs}.
		
		We aim to prove \eqref{eq:smoother-stability-proof4} using Jensen's operator inequality. To that end, we need to scale $R^s_k$, so that \eqref{eq:jensen-inequality} is applicable. Using the assumed stable decomposition of $v$, we note that
		\algns{
			\sum_{\nu \in \N_k}\inner{Q_{k,\nu}v}{Q_{k,\nu} v} &= \sum_{\nu \in \N_k} \inner{Q_{k,\nu} v}{v} \\
			&= \sum_{\nu, \eta \in \N_k}\inner{Q_{k,\nu} v}{v_\eta} \\
			&\leq \sum_{\nu, \eta \in \N_k}\norm{Q_{k,\nu} v}\norm{v_\eta}_{\Omega_{k,\nu}} \\
			&\leq \sum_{\eta \in \N_k}\left( \sum_{\nu \in \N_k} \norm{Q_{k,\nu}v}^2 \right)^\frac{1}{2}\left(\sum_{\nu \in \N_k}\norm{v_\eta}_{\Omega_{k,\nu}}^2\right)^\frac{1}{2}.
		} 
		Because of the shape-regularity of $\T_k$, we can bound
		\eqns{
			\sum_{\nu \in \N_k}\norm{v_\eta}_{\Omega_{k,\nu}}^2 \leq K_1 \norm{v_\eta}^2,
		}
		for some $K_1$ that is independent of $\eta$ and $k$.
		Thus,
		\algns{
			\sum_{\nu \in \N_k}\inner{Q_{k,\nu}v}{Q_{k,\nu} v} &\leq K_1 \sum_{\eta \in \N_k} \norm{v_\eta}^2 \\
			&\leq K_1 K_0 \norm{v}^2.
		}
		If we now define $\tilde{Q}_{k,\nu} = (K_0K_1)^{-\frac{1}{2}}Q_{k,\nu}$, and $\tilde{R}^s_k = \inv{(K_0 K_1)}R^s_k$, we have that
		\eqns{
			\sum_{\nu \in \N_k}\inner{\tilde{Q}_{k,\nu}v}{\tilde{Q}_{k,\nu} v} \leq \norm{v}^2 \text{ and } \inner{\tilde{R}^s_k v}{v} = \sum_{\nu \in \N_k} \inner{A^{-s}_{k,\nu}\tilde{Q}_{k,\nu} v}{\tilde{Q}_{k,\nu} v}.
		}
		We can now use Jensen's operator inequality \eqref{eq:jensen-inequality}, together with an argument analogous to that in the proof of Lemma \ref{lem:subspace-As-estimate} to get
		\algns{
			R^s_k &= K_0K_1 \tilde{R^s_k} \\
			&\leq K_0 K_1 (\tilde{R}^1_k)^s \\
			&= (K_0 K_1)^{1-s}(R^1_k)^s
		}
		This, together with \eqref{eq:smoother-stability-proof3}, proves the second inequality of \eqref{eq:smoother-stability-BPXs} with $C_2 = (K_0K_1)^{1-s}C^s$.
	\end{proof}
	
	We observe that the proof of Lemma \ref{lem:smoother-stability-verification}, shows that if the decomposition $V_k = \sum_{\eta \in \N_k}V_{k,\nu}$ is stable in both $L^2$- and $H_0^1$-norms, then it is also stable in the fractional norm $\norm{\cdot}_{s,k}$. That is, if there are constants $c_0, c_1 > 0$ so that
	\eqns{
		\inner{A^s_k v}{v} \leq c_s\inner{\inv{(R^s_k)}v}{v}, \quad \forall v\in V_k,
	}
	with $s=0$ and $s=1$, then the same holds for every $s\in [0,1]$, with $c_s = c_0^{1-s}c_1^s$. In this way, the smoother defined by \eqref{eq:Rk-definition} is the natural interpolation between the corresponding smoothers for $s=0$ and $s=1$.
	
	In our current context, where $V_k$ is the space of continuous, piecewise linear functions relative to $\T_k$, the assumption in Lemma \ref{lem:smoother-stability-verification} can be verified in the following manner.
	Fix $k$ and $v \in V_k$. Let $\setof{\theta_\nu}_{\nu \in \N_k}$ be a partition of unity subordinate to $\setof{\Omega_{k,\nu}}_{\nu \in \N_k}$, and let $\pi_k$ denote the nodal interpolant on $V_k$. We then set $v_\nu = \pi_k \theta_\nu v \in V_{k,\nu}$, in which case we see that $v = \sum_{\nu \in \N_k} v_\nu$.
	Furthermore, we have that
	\eqn{
		\label{eq:CG1-L2-stable-decomposition}
		\algnd{
			\sum_{\nu \in \N_k}\norm{v_\nu}^2 &=\sum_{\nu \in \N_k} \int_{\Omega_{k,\nu}} |v_\nu|^2 \intd x \\
			&\leq \sum_{\nu \in \N_k} \norm{v}^2_{\Omega_{k,\nu}} \\
			&\leq (n+1)\norm{v}^2, 
		}
	}
	where the last inequality follows by the fact that no point in $\Omega$ is contained in more than $n+1$ subdomains $\Omega_{k,\nu}$. Thus, \eqref{eq:L2-stable-decomposition} holds with $K_0 = n+1$.
	
	As noted in \cite[Remark 5.1]{bramble1987newconvergence} the $\alpha$ in the approximation and regularity assumption \eqref{ass:approximation-property} is closely related to the elliptic regularity of the continuous problem. Therefore, we make the following assumption:
	\begin{assumption}
		\label{ass:elliptic-regularity}
		There is an $\alpha \in (0,1]$ so that $A$ is a bounded operator from $H^1_0(\Omega) \bigcap H^{1+\alpha}(\Omega)$ to $H^{-1+\alpha}(\Omega)$, and  $A^{-1}$ is a bounded operator from $H^{-1+\alpha}(\Omega)$ to $H^1_0(\Omega) \bigcap H^{1+\alpha}(\Omega)$.
	\end{assumption}
	Assumption \ref{ass:elliptic-regularity} is standard for proving condition \eqref{ass:approximation-property} in the case of $s=1$ (cf. for instance \cite{bramble1987newconvergence}). In \cite[Thm 4.3, and Rem. 4.1]{bonito2015numerical} Bonito et al. used Assumption \ref{ass:elliptic-regularity} to prove the error estimate 
	\eqns{
		\label{eq:bonito-error-estimate}
		\norm{(A^{-s}-A_k^{-s}Q_k)f} \leq  C h_k^{2s}\norm{f}, \quad \forall f \in L^2(\Omega),
	}
	when $\alpha > s$. By the triangle inequality and the bounded refinement hypothesis it then follows that
	\eqn{
		\label{eq:bonito-error-estimate2}
		\norm{(A^{-s}_{k}-A_{k-1}^{-s}Q_{k-1})f} \leq  C h_k^{2s}\norm{f}, \quad \forall f \in V_k,
	}
	for each $k$.
	This estimate is sufficient to verify \eqref{ass:approximation-property} in our current context. The result is stated in the following Lemma.
	\begin{lemma}
		\label{lem:approximation-property}
		Assume that Assumption \ref{ass:elliptic-regularity} is satisfied with $\alpha > s$. Then there is a constant $C_3 > 0$, so that for every $k$
		\eqn{
			\label{eq:approximation-property-BPXs}
			\inner{A^s_k(I-P^s_{k,k-1})v}{v} \leq C_{3} \frac{\norm{A^s_k v}^2}{\lambda^s_k}.
		}
	\end{lemma}
	\begin{proof}
		From the definition of $P^s_{k,k-1}$ in \eqref{eq:Ps_k-definition}, 
		\eqns{
			I - P^s_{k,k-1} = I - A^{-s}_{k-1}Q_{k-1}A^{s}_k = (A^{-s}_k - A^{-s}_{k-1}Q_{k-1})A^s_k,
		}
		and so, for any $v \in V_k$
		\eqns{
			\inner{A^s_k(I-P^s_{k,k-1})v}{v} \leq \inner{(A^{-s}_k - A^{-s}_{k-1}Q_{k-1})A^s_k v}{A^s_k}.
		}
		Using Cauchy-Schwarz inequality together with the error estimate \eqref{eq:bonito-error-estimate2}, we get
		\eqn{
			\label{eq:approximation-property-BPXs-proof1}
			\inner{A^s_k(I-P^s_{k,k-1})v}{v} \leq \norm{(A^{-s}_k - A^{-s}_{k-1}Q_{k-1})A^s_k v} \norm{A^s_k v}  \leq C h_k^{2s} \norm{A^s_k v}^2.
		}
		By the quasi-uniformity of the mesh and $h_k^2 \leq C \lambda_k^{-1}$ it follows that $h_k^{2s} \leq C\lambda_k^{-s}$ . Using this in \eqref{eq:approximation-property-BPXs-proof1} completes the proof.
	\end{proof}
	
	We are finally in a position to prove the main theorem of this section.
	\begin{theorem}
		\label{thm:BPXs-spectral-equivalence}
		Let Assumption \ref{ass:elliptic-regularity} be satisfied with $\alpha > s$. Then, for $s \in [0,1]$ with $B_h^s$ defined by \eqref{eq:Rk-definition} and \eqref{eq:BPXs-preconditioner-definition} satisfies
		\eqn{
			\label{eq:BPXs-spectral-equivalence}
			C_1 \inv{C_3}\inner{A_h^s v}{v} \leq \inner{B_h^s A_h^s v}{A_h^s v} \leq C_2J\inner{A_h^s v}{v}, \quad \forall v \in V,
		}
		where $C_1$, $C_2$, and $C_3$ are the same as in Lemmas \ref{lem:smoother-stability-verification} and \ref{lem:approximation-property}.
	\end{theorem}
	
	\begin{proof}
		This result is a straightforward application of Theorem \ref{thm:bpx-spectral-equivalence} together with Lemmas \ref{lem:smoother-stability-verification} and \ref{lem:approximation-property}.
	\end{proof}
	Theorem \ref{thm:BPXs-spectral-equivalence} shows that the condition number $K(B_h^s A_h^s) \leq C J$, and so increases linearly with the number of mesh levels, but is independent of $h$.
	
	With less regularity of the domain, we can still prove a slightly weaker form of spectral equivalence. By the assumed quasi-uniformity of $\T_k$, we have for $k=2,\ldots,J$ that
	\eqns{
		\norm{(I-Q_{k-1})v}^2 \leq C h_k^2 \norm{v}_{1}^2, \quad \forall v \in V_k.
	}
	This, together with the boundedness of $I-Q_{k-1}$ and interpolation theory, yields
	\eqns{
		\norm{(I-Q_{k-1})v}^2 \leq C h_k^{2s}\norm{v}_{s,k}^2 \leq C_4 \lambda_k^{-s}\norm{v}_{s,k}^2,
	}
	for some constant $C_4$, independent of $k$.
	By the discussion in Remark \ref{rem:bpx-weak-spectral-equivalence}, we get that
	\eqn{
		\label{eq:BPXs-weak-spectral-equivalence}
		\inv{C_4}C_1 \inv{J}\inner{A^s_h v}{v} \leq \inner{B^s_h A^s_h v}{A^s_h v} \leq C_2 J\inner{A^s_h v}{v}, \quad \forall v \in V_h,
	}
	and the condition number is bounded by $K(B^s_h A^s_h) \leq C J^2$.

% PRECONDITIONING FOR NEGATIVE s:
	\section{Preconditioner when $s \in [-1,0]$}
	\label{sec:preconditioning-negative-s}
	For $s \in [-1,0]$, the large eigenvalues of $A^s_h$ correspond to smooth functions. In a multilevel setting this means that neither relaxation nor coarse grid correction will reduce the oscillatory components of the error. As such, we cannot expect a direct multigrid approach to work. Moreover, when $s <0$ the Löwner-Heinz' and Jensen's operator inequalities in \eqref{eq:Loewner-Heinz} and \eqref{eq:jensen-inequality} fail to hold, and the argument of Section \ref{sec:discrete-bpx} is no longer valid. In this section, we will therefore investigate an alternative approach for constructing  preconditioners.
	
	We will base the preconditioner for $A^s_h$ when $s$ is negative on our previously defined preconditioners $B^t_h$ for $t \in [0,1]$ together with the multiplicative decomposition 
	of $A_h$  
	\eqn{
		\label{eq:As-neg-decomp}
		A^{-s}_h = A_h^{-\frac{1+s}{2}}A_h A_h^{-\frac{1+s}{2}}.
	}
	We have for every $u \in V_h$ and $t \in \reals$ that
	\eqns{
		\norm{u}_{-s,h} = \norm{A^{-\frac{t+s}{2}}_hu}_{t,h}.
	}
	The specific form we will use below is  
	\eqn{
		\label{eq:As-neg-mapping-property}
		\norm{u}_{-\frac{1+s}{2}+\beta,h} = \norm{A^{-\frac{1+s}{2}}_h u}_{\frac{1+s}{2}+\beta,h},
	}
	which is valid for any $\beta \in \reals$.
	
	Replacing the left- and rightmost factor of the right hand side in \eqref{eq:As-neg-decomp} with a spectrally equivalent preconditioner 
	$B^{\frac{1+s}{2}}_h$, yields a symmetric positive definite operator
	\eqn{
		\label{eq:tildeBs-def}
		\tilde{B}^{s}_h := B^{\frac{1+s}{2}}_h A_h B^{\frac{1+s}{2}}_h.
	}
	
	We want $\tilde{B}^s_h$ to be spectrally equivalent to $A^{-s}_h$. That is, there exist constants $C_1, C_2$ so that for every $u \in V_h$,
	\eqn{
		\label{eq:tildeBs-spectral-equivalence}
		C_1\inner{A^s_h u}{u} \leq \inner{\tilde{B}^s_h A^s_h u}{A^s_h u} \leq C_2 \inner{A^s_h u}{u}
	}
	holds. By the definition of $\tilde{B}^s_h$,
	\eqns{
		\inner{\tilde{B}^s_h A^s_h u}{A^s_h u} = \inner{A_hB^{\frac{1+s}{2}}_hA^s_h u}{B^{\frac{1+s}{2}}_hA^s_h u} = \norw{B^{\frac{1+s}{2}}_hA^s_h u}{1}^2,
	}
	and since $\inner{A^s_h u}{u} = \norw{u}{s,h}^2$, we see that the spectral equivalence in \eqref{eq:tildeBs-spectral-equivalence} is equivalent to
	\eqn{
		\label{eq:tildeBs-norm-equivalence}
		C_1^{\frac{1}{2}}\norm{u}_{s,h} \leq \norm{B_h^{\frac{1+s}{2}}A^s_h u}_1 \leq C_2^{\frac{1}{2}} \norm{u}_{s,h}, \quad \forall u \in V_h.
	}
	Using the preconditioner described in section \ref{sec:discrete-bpx}, we have by the spectral equivalence established in Theorem \ref{thm:BPXs-spectral-equivalence} that there are constant $C_1, C_2 > 0$ so that
	\eqn{
		\label{eq:Bs-norm-spectral-equivalence}
		C_1 \norm{u}_{-\frac{1+s}{2},h} \leq \norm{B_h^{\frac{1+s}{2}}u}_{\frac{1+s}{2},h} \leq C_2 J \norm{u}_{-\frac{1+s}{2},h}, \quad u \in V_h.
	}

	We assume now some additional regularity on $B_h^{\frac{1+s}{2}}$, similar to \eqref{eq:As-neg-mapping-property}. That is, for some $\beta$, we have the norm equivalence
	\eqn{
		\label{eq:tildeBs-norm-equivalence-regularity}
		C_1 \norm{u}_{-\frac{1+s}{2}+\beta,h} \leq \norm{B_h^{\frac{1+s}{2}}u}_{\frac{1+s}{2}+\beta,h} \leq C_2 J \norm{u}_{-\frac{1+s}{2}+\beta,h}.
	}
	
	In particular, with $\beta = \frac{1-s}{2} \in \left[\frac{1}{2}, 1 \right]$, and replacing $u$ by $A^s_h u$  in \eqref{eq:tildeBs-norm-equivalence-regularity} we recover \eqref{eq:tildeBs-norm-equivalence} and the spectral equivalence \eqref{eq:tildeBs-spectral-equivalence}.
	We note also that if we assume the additional regularity of \eqref{eq:tildeBs-norm-equivalence-regularity}, we can bound the condition number of $\tilde{B}^s_h A^s_h$, $K(\tilde{B}^s_hA^s_h)$, by
	\eqn{
		\label{eq:tildeBs-condition-number-bound}
		K(\tilde{B}^s_h A^s_h) \leq K(B_h^{\frac{1+s}{2}}A^{\frac{1+s}{2}}_h)^2.
	}

% IMPLEMENTATIONAL CONCERNS:
	\section{Implementational concerns}
	\label{sec:implementation}
	The discrete operators discussed so far are related to, but are not the same as the matrices used in the implementation.  
	%So far, the discrete fractional Laplacian, $A^s_h$, and the preconditioner, $B^s_h$, developed in section \ref{sec:discrete-bpx}, have been discussed as operators on $V_h$. 
	In this section we will discuss how to implement these operators.
	We begin by discussing the matrix representation of the discrete fractional operators. 
	We refer also to \cite{mardal2011preconditioning} for more details. 
	While 
	the discrete fractional operators satisify the group property $A_h^s A_h^t = A_h^{s+t}$, their
	matrix representations do not. In particular, for $t=-s$, $A_h^s A_h^{-s} = I_h$ and the finite element matrix
	representation of the identity is the mass matrix. Hence, in order to provide a precise description 
	of the interpolation of the involved matrices, we let $\setof{ \phi_h^i}_{i=1}^{N_h}$ be the standard nodal basis for $V_h$, and we introduce the operators   
	$\pi_h, \mu_h: V_h\to \reals^{N_h}$, defined by
	\eqn{
		\label{eq:matvec-representations}
		\algnd{
			v &= \sum_{i=1}^{N_h}\left(\pi_h v\right)_i \phi^i_h, \quad \text{ and } \\
			(\mu_h v)_i &= \inner{v}{\phi^i_h}, \quad i=1,\ldots,N_h.
		}
	}
	Subsequently, we will refer to $\pi_h v$ and $\mu_h v$ as the primal- and dual vector representation of $v$, respectively. The
	primal representation is sometimes called the nodal representation. 
	We then have that
	\eqn{
		\label{eq:vector-representation-adjoints}
		\mu_h^* = \pi_h^{-1}, \quad \text{ and } \quad  \pi_h^* = \mu_h^{-1}
	}
	To see this, take $\mat{v} \in \reals^{N_h}$, and $u \in V_h$. Then,
	\eqns{
		\algnd{
			\inner{\mu_h^* \mat{v}}{u} &= \inner{\mat{v}}{\mu_h u}_{l^2} \\
			&= \sum_{i=1}^{N_h} \mat{v}_i \inner{u}{\phi_h^i} \\
			&= \inner{u}{\sum_{i=1}^{N_h}\mat{v}_i\phi_h^i} \\
			&= \inner{u}{\pi_h^{-1}\mat{v}},
		}
	}
	where $\inner{\cdot}{\cdot}_{l^2}$ is the standard Euclidean inner product on $\reals^{N_h}$. This proves the first identity in \eqref{eq:vector-representation-adjoints}. The second identity is proven similarly.
	
	Using these operators, the stiffness matrix can then be expressed as  
	\[
	\mat{A}_h = \mu_h A_h \pi_h^{-1}, \quad \mbox{and} \quad (\mat{A}_h)_{i,j} = (A_h \phi_h^j, \phi_h^i ), \quad 1\le i, j \le N_h,   
	\]
	and the mass matrix is 
	\[
	\mat{M}_h = \mu_h I_h \pi_h^{-1} = \mu_h \pi_h^{-1}, \quad \mbox{and} \quad (\mat{M}_h)_{i,j} = (M_h \phi_h^j, \phi_h^i ), \quad 1\le i, j \le N_h .   
	\]
	We see that for both the stiffness- and mass matrix, a matrix-vector product takes primal vectors as input and returns dual vectors.
	
	For the matrix realization of $A^s_h$, let $\setof{(\lambda_i, \mat{u}_i)}_{i=1}^{N_h} \subset \reals \times \reals^{N_h}$ be the eigenpairs of the generalized eigenvalue problem
	\eqns{
		\label{eq:generalized-eigenvalue-problem-matrices}
		\mat{A}_h \mat{u}_i = \lambda_i \mat{M}_h \mat{u}_i,
	}
	normalized so that $\mat{u}_j^{\top} \mat{M}_h \mat{u}_i = \delta_{i,j}$. Setting $\mat{\Lambda}_h = \operatorname{diag} (\lambda_1,\ldots,\lambda_{N_h})$, and $\mat{U} = [\mat{u}_1,\ldots, \mat{u}_{N_h}]$, we have that
	\eqn{
		\label{eq:orthonormal-eigenvectors}
		\mat{U}^{\top} \mat{M}_h \mat{U} = I, \quad \text{ and } \quad \mat{U}^{\top}\mat{A}_h \mat{U} = \mat{\Lambda}_h.
	}
	We then define
	\eqn{
		\label{eq:Ash-matrix-realization}
		\mat{A}_h^s = \left(\mat{M}_h \mat{U} \right)\mat{\Lambda}_h^s \left(\mat{M}_h \mat{U} \right)^{\top}.
	}
	One can verify that the entries of $\mat{A}^s_h$ satisfy
	\eqns{
		(\mat{A}_h^s)_{i,j} = \inner{A^s_h \phi_h^j}{\phi_h^i},
	}
	in which case $\mat{A}^s_h = \mu_h A^s_h \pi_h^{-1}$.
	%\MK{Trygve, I changed the following 2 senteces:}
	Using \eqref{eq:orthonormal-eigenvectors} we are also able to see that
	\eqn{
		\label{eq:Ash-matrix-inverse}
		\inv{(\mat{A}_h^s)} = \mat{U}\mat{\Lambda}_h^{-s} \mat{U}^{\top} = \pi_h A^{-s}_h \mu_h^{-1},
	}
	making it the matrix realization of $A^s_h$ viewed as an operator from $X_h$ to $\dual{X_h}$. However, the group properties 
	mentioned above only make sense when we consider $A^s_h$ as operators on $X_h$. Thus, we see that for the matrices
	\eqns{
		\pi_h A^s_h \inv{\pi_h} = (\pi_h \inv{\mu_h}) \mu_h A^s_h \inv{\pi_h} = \inv{\mat{M}_h} \mat{A}_h^s,
	}
	the group properties are satisfied. This can also be verified using the definition of $\mat{A}^s_h$ in \eqref{eq:Ash-matrix-realization}.
	
	Since matrix-vector products involving $\mat{A}^s_h$ take primal vectors as input and return dual vectors, the matrix realization of $B^s_h$ should take dual vectors as input and return primal vectors. Then the product $\mat{B}^s_h \mat{A}^s_h$ acts on primal vectors, and is thus suitable for a Krylov subspace method. See also \cite[Section 6]{mardal2011preconditioning} and \cite[Section 15]{bramble1993multigrid}. Therefore, we define
	\eqn{
		\label{eq:matBs-def}
		\mat{B}^s_h = \pi_h B^s_h \mu_h^{-1}.
	}
	
	To see how $\mat{B}^s_h$ is implemented, we begin by supposing that $\dim V_k = N_k$ for $k=1,\ldots,J$. Let $\setof{\phi_k^i}_{i=1}^{N_k}$ be bases for $V_k$, and we define operators $\pi_k, \mu_k : V_k \to \reals^{N_k}$ analogously to \eqref{eq:matvec-representations}. We then define mass- and stiffness matrices on level $k$ as $\mat{M}_k = \mu_k \pi_k^{-1}$ and $\mat{A}_k = \mu_k A_k \pi_k^{-1}$, respectively.
	
	By assumption, for every $k$, $V_k \subset V_h$, and so there are matrices $\mat{I}_k: \reals^{N_k} \to \reals^{N_h}$ so that
	\eqns{
		\phi_k^i = \sum_{j=1}^{N_h} (\mat{I}_k)_{i,j}\phi_h^j, \quad i=1,\ldots,N_k.
	}
	In fact, $\mat{I}_k$ is the matrix realization of the inclusion operator $I_k: V_k \to V_h$, i.e. $\mat{I}_k = \pi_h I_k \pi_k^{-1}$. Using that $Q_k = I_k^*$ and \eqref{eq:vector-representation-adjoints} we have that the transpose of $\mat{I}_k$ satisfies
	\algns{
		\mat{I}_k^{\top} &= \left(\pi_h I_k \pi_k^{-1} \right)^* \\
		&= (\pi_k^{-1})^* Q_k \pi_h^* \\
		&= \mu_k Q_k \mu_h^{-1} \\
		&=: \mat{Q}_k,
	}
	which is the matrix realization of $Q_k$ in dual representation. Thus, for the matrix $\mat{B}^s_h$ we have that
	\eqn{
		\label{eq:matBs-impl-step1}
		\algnd{
			\mat{B}^s_h &=   \pi_h B_h^s \mu_h^{-1} \\ 
			&= \sum_{k=1}^J \pi_h I_k R_{k}^{s}  Q_k \mu_h^{-1} \\ 
			&= \sum_{k=1}^J (\pi_h I_k \pi_k^{-1}) (\pi_k R_k^s \mu_k^{-1}) (\mu_k Q_k \mu_h^{-1}) \\
			&= \sum_{k=1}^J \mat{Q}_k^{\top} \mat{R}^s_k \mat{Q}_k,
		}
	}
	where we define $\mat{R}^s_k = \pi_k R_k^s \mu_k^{-1}$ as the matrix realization of $R^s_k$. We see that due to \eqref{eq:Ash-matrix-inverse} $\mat{R}^s_1 = (\mat{A}^s_1)^{-1}$. For $k\geq2$ we define for $\nu \in \N_k$ operators $\pi_{k,\nu}, \mu_{k,\nu}: V_{k,\nu} \to \reals^{\dim V_{k,\nu}}$ and matrices $\mat{Q}_{k,\nu}: \reals^{N_k} \to \reals^{\dim V_{k,\nu}}$, similarly to the above. The matrix realization of $R^s_k$ then becomes
	\eqn{
		\label{eq:matRs-impl}
		\mat{R}^s_k = \sum_{\nu \in \N_k} \mat{Q}_{k,\nu}^{\top} (\mat{A}^s_{k,\nu})^{-1} \mat{Q}_{k,\nu}.
	}
	Here, $\mat{A}^s_{k,\nu} = \mu_{k,\nu} A^s_{k,\nu} \pi_{k,\nu}^{-1}$. By \eqref{eq:Ash-matrix-inverse}, the implementation of $\mat{R}^s_k$ will require solving many small eigenvalue problems. In the particular case of continuous, piecewise linear finite element functions, and subdomains $\Omega_{k,\nu}$ as described in Section \ref{sec:discrete-bpx}, the subspaces $V_{k,\nu}$ are one-dimensional. The matrix $\mat{R}^s_k$ is then diagonal, with entries
	\eqns{
		(\mat{R}^s_k)_{i,i} = \frac{1}{(\mat{M}_k)_{i,i}^{1-s}(\mat{A}_k)_{i,i}^s}, \quad i=1,\ldots,N_k.
	}
	That is, this is the smoother mentioned in the introduction. 
	
	Inserting \eqref{eq:matRs-impl} into \eqref{eq:matBs-impl-step1} we get
	\eqn{
		\label{eq:matBs-impl-final}
		\mat{B}^s_h = \mat{Q}_1^{\top} (\mat{A}^s_1)^{-1} \mat{Q}_1 + \sum_{k=2}^{J}\mat{Q}_k^{\top}\left(\sum_{\nu \in \N_k} \mat{Q}_{k,\nu}^{\top} (\mat{A}^s_{k,\nu})^{-1} \mat{Q}_{k,\nu} \right)\mat{Q}_k.
	}
	%\begin{rem}
	%	In the implementation of $B_h^s$, it is simplest to view it as an operator from $\dual{V_h}$ to $V_h$. In this case, the implementations of the projections $Q_k$ and $Q_{k,\nu}$ are nothing more than suitable restrictions of the input vector. The smoothers $R^s_k$ can then be implemented by restriction of the input vector, solving a small eigenvalue problem and multiplying by $A^{-s}_{k,\nu} = \inv{(A^s_{k,\nu})}$ using \eqref{eq:Ash-matrix-inverse}, and adding it to the result vector in the appropriate place.
	%\end{rem}
	We end this section by showing how to implement $\tilde{B}^s_h$, when $s\in[-1,0]$. In this case, the matrix realization of $\tilde{B}^s_h$ can be found from $\mat{B}^{\frac{1+s}{2}}_h$ and $\mat{A}_h$ by
	\eqn{
		\label{eq:mattildeBs-impl}
		\algnd{
			\mat{\tilde{B}}^s_h &:= \pi_h \tilde{B}^s_h \mu_h^{-1} \\
			&= (\pi_h B^{\frac{1+s}{2}}_h \mu_h^{-1}) (\mu_h A_h \pi_h^{-1})(\pi_h B^{\frac{1+s}{2}}_h \mu_h^{-1}) \\
			&= \mat{B}^{\frac{1+s}{2}}_h \mat{A}_h \mat{B}^{\frac{1+s}{2}}_h.
		}
	}
	That is, $\tilde{B}^s_h$ is implemented as an application of $B^{\frac{1+s}{2}}_h$, followed by a multplication of the stiffness matrix and a second application of $B^{\frac{1+s}{2}}_h$.
	
% NUMERICAL TESTS:
	\section{Numerical experiments}
	\label{sec:numerical-experiments}
	% FILE: numerical_experiments.tex
	%!TeX root=fractional_mg_main.tex
	In this section we present a series of numerical experiments that aim to validate the theoretical results we established in previous sections. We also present numerical results for the case when $s < 0$, using $\tilde{B}^s_h$, defined in \eqref{eq:tildeBs-def}, as preconditioner. Specifically, in Section \ref{sec:numerical-experiments-fracLap} we solve
	\eqns{
		A^s_h u = f,
	}
	using preconditioned conjugate gradient method with $B^s_h$ defined in \eqref{eq:BPXs-preconditioner-definition} as preconditioner. In Section \ref{sec:numerical-experiments-EMI}, we consider a coupled multidomain problem, where the weakly imposed continuity on the interface leads to a Lagrange multiplier in $H^{\pm\frac{1}{2}}$.
	
	The numerical experiments are conducted using random initial guess. Convergence in the iterative methods used is reached when the relative preconditioned residual, i.e. $\frac{\inner{Br_k}{r_k}}{\inner{Br_0}{r_0}}$, where $r_k$ is the residual at the $k$-th iteration and $B$ is the preconditioner, is below a given tolerance.
	\subsection{Preconditioning the fractional Laplacian}
	\label{sec:numerical-experiments-fracLap}
	In the first set of numerical experiments, we show the performance of the preconditioners $B^s_h$ and $\tilde{B}^s_h$, defined in \eqref{eq:BPXs-preconditioner-definition} and \eqref{eq:tildeBs-def}, respectively,
	depending on the sign of $s$ for the $A^s_h$ inner product. That is, for a given $f_h \in V_h$, we solve: Find $u_h \in V_h$ such that
	\eqn{
		\label{eq:numerex1}
		\inner{A^s_h u_h}{v} = \inner{f_h}{v}, \quad \forall v \in V_h,
	}
	where $s \in [-1,1]$. We take $\Omega = [0,1] \subset \reals$, and $\T_h$ is a uniform partition of $\Omega$ consisting of $N = \frac{1}{h}$ elements. $V_h$ is then the space of continuous, piecewise linear functions relative to $\T_h$ that vanish on $\pd \Omega$. We solve the linear system arising from \eqref{eq:numerex1} using preconditioned conjugate gradient, with $B^s_h$ as preconditioner if $s \geq 0$, and $\tilde{B}^s_h$ if $s < 0$. For $s \geq 0$, iteration counts and estimated condition numbers can be viewed in Table \ref{tab:fraclap_positive_s_bpx}. From these results we see that both the iteration counts and condition numbers stay uniformly bounded for each $s$.
	
	The analogous results for $s \leq 0$ can be seen in Table \ref{tab:fraclap_negative_s_bpx}. Here the situation is slightly more complicated. For each $s$, the iteration counts and condition numbers seem to increase for small $N$ (large $h$), but ultimately stay bounded when $N$ is increased. Worth noting is that the bound \eqref{eq:tildeBs-condition-number-bound} is relatively sharp. For instance, for $s = -1$, the preconditioner $\tilde{B}^s_h$ does two applications of $B^0_h$, and has estimated condition numbers around $193$. By Table \ref{tab:fraclap_positive_s_bpx}, $K(B^0_hA^0_h) \approx 13.8$, and so $K(\tilde{B}^{-1}A^{-1}_h) \approx K(B^0_hA^0_h)^2$. Similar relations holds for other values of $s \leq 0$.
	\begin{table}[t]
		\begin{tabular}{l | l l l l l }
			\diagbox{$s$}{$N$} & 32& 64& 128& 256& 512\\
			\hline
			$0.0$& $20(13.5)$& $25(13.6)$& $28(13.8)$& $29(13.8)$& $29(13.9)$\\ 
			$0.1$& $18(8.7)$& $21(8.9)$& $23(8.9)$& $24(8.9)$& $24(8.9)$\\ 
			$0.2$& $16(5.8)$& $18(6.4)$& $19(6.5)$& $21(6.5)$& $21(6.6)$\\ 
			$0.3$& $14(4.2)$& $15(4.7)$& $17(4.9)$& $18(5.0)$& $18(5.0)$\\ 
			$0.4$& $12(3.4)$& $14(3.7)$& $15(3.8)$& $15(3.9)$& $16(3.9)$\\ 
			$0.5$& $11(2.9)$& $12(3.0)$& $13(3.1)$& $13(3.1)$& $14(3.2)$\\ 
			$0.6$& $12(2.9)$& $13(3.0)$& $13(3.0)$& $14(3.1)$& $14(3.0)$\\ 
			$0.7$& $12(3.0)$& $13(3.0)$& $14(3.1)$& $14(3.1)$& $14(3.1)$\\ 
			$0.8$& $13(3.2)$& $14(3.3)$& $14(3.3)$& $14(3.3)$& $14(3.3)$\\ 
			$0.9$& $14(3.5)$& $15(3.6)$& $15(3.6)$& $15(3.6)$& $15(3.6)$\\ 
			$1.0$& $14(4.0)$& $16(4.1)$& $16(4.1)$& $16(4.1)$& $16(4.1)$\\ 
		\end{tabular}
		\caption{Numerical results for $(-\Delta)^s$ with nonnegative $s$. Table shows the number 
			of preconditioned conjugate gradient iterations until reaching error tolerance $10^{-15}$. Estimated condition numbers are shown inside parentheses. $N$ is number of elements on the finest mesh. $J = 5$ in all tests.}
		\label{tab:fraclap_positive_s_bpx}
	\end{table}
	\begin{table}[t]
		\begin{tabular}{l | l l l l l }
			\diagbox{$s$}{$N$} & 32& 64& 128& 256& 512\\
			\hline
			$-1.0$& $32(184.4)$& $47(192.4)$& $56(192.7)$& $64(193.8)$& $62(191.2)$\\ 
			$-0.9$& $28(119.0)$& $43(118.9)$& $50(120.5)$& $54(120.7)$& $55(119.9)$\\ 
			$-0.8$& $26(78.3)$& $37(82.6)$& $46(84.5)$& $48(83.8)$& $49(83.9)$\\ 
			$-0.7$& $25(53.0)$& $33(60.1)$& $40(61.9)$& $42(62.1)$& $45(61.5)$\\ 
			$-0.6$& $24(36.9)$& $31(43.8)$& $35(45.8)$& $38(46.2)$& $41(46.2)$\\ 
			$-0.5$& $22(26.8)$& $25(31.9)$& $30(34.3)$& $34(34.9)$& $38(35.1)$\\ 
			$-0.4$& $20(20.4)$& $24(24.8)$& $28(26.5)$& $32(27.0)$& $37(27.1)$\\ 
			$-0.3$& $17(16.1)$& $21(19.3)$& $27(20.7)$& $30(21.1)$& $34(21.1)$\\ 
			$-0.2$& $17(13.1)$& $21(15.3)$& $25(16.4)$& $29(16.7)$& $32(16.7)$\\ 
			$-0.1$& $16(11.0)$& $20(12.4)$& $23(13.2)$& $27(13.5)$& $29(13.5)$\\ 
			$0.0$& $14(9.4)$& $17(10.4)$& $20(11.0)$& $24(11.2)$& $27(11.1)$\\ 
		\end{tabular}
		\caption{Numerical results for $(-\Delta)^s$ with negative $s$. Table shows the number of preconditioned conjugate gradient iterations until reaching error tolerance $10^{-15}$. Estimated condition numbers are shown inside parentheses. $N$ is number of elements on the finest mesh. $J = 5$ in all tests.}
		\label{tab:fraclap_negative_s_bpx}
	\end{table}

	\subsection{Multidomain preconditioning}
	\label{sec:numerical-experiments-EMI}
	In this section we apply the multilevel algorithm \eqref{eq:BPXs-preconditioner-definition}
	to construct mesh independent preconditioners for a coupled multidomain problem
	originating from a geometrically accurate model of electric signal propagation
	in cardiac tissue, the EMI model, \cite{tveito2017cell}.%, see also \cite{agudelo2013computationally}. 
	We remark that the EMI model is simple in a sense that it is a single-physics problem where two elliptic equations  
	are coupled. However, the interface problems encountered here are identical to those 
	found in multiphysics applications, e.g. the couled Darcy-Stokes system \cite{layton2002coupling} or 
	the Stokes-Biot system \cite{Ambartsumyan2018}.
	
	%\KAM{For Miro: Vi burde ha med noen ref til liknede anvendelser introduksjonen til denne seksjonen }
	
	Let $\Omega\subset\mathbb{R}^2$ be a bounded domain decomposed into two
	non-overlapping subdomains $\Omega_1$, $\Omega_2$ with a common interface
	$\Gamma=\partial\Omega_1\bigcap\partial\Omega_2$ forming a closed curve.
	Motivated by the application the subdomain $\Omega_1$ is designated as the
	\emph{exterior} domain, i.e. $\partial\Omega_2\bigcap\partial\Omega=\emptyset$.
	With $\epsilon > 0$ and $n$ the outer normal of the exterior domain we now
	aim to solve 
	\begin{equation}
	\label{eq:emi_system}
	\begin{aligned}
	u_1-\Delta u_1 &= f_1, & x\in \Omega_1,\\
	u_2 -\Delta u_2&= f_2, & x\in \Omega_2,\\
	n\cdot \nabla u_1 - n\cdot\nabla u_2 &= 0, & x\in \Gamma,\\
	\epsilon (u_1 - u_2) + n\cdot \nabla u_1 &= g, & x\in\Gamma.
	\end{aligned}
	\end{equation}
	The choice of boundary conditions for \eqref{eq:emi_system} shall be discussed
	shortly. We remark that in the EMI model the parameter $\epsilon$ plays a role
	of inverse time step and thus algorithms robust with respect to the parameter
	are of interest. However, here the system will be considered only for a fixed
	choice of the parameter. 
	
	Considering \eqref{eq:emi_system} with homogeneous Neumann boundary conditions
	$n\cdot\nabla u_1=0$ on $\partial\Omega$ and letting $W_1=H^1(\Omega_1)\times H^1(\Omega_2)\times (H^{-1/2}(\Gamma)\bigcap \epsilon^{-1/2}L^2(\Gamma))$
	the variational formulation of \eqref{eq:emi_system} defines an operator
	$\mathcal{A}_1: W_1\rightarrow W'_1$
	\begin{equation}\label{eq:mortarA}
	\mathcal{A}_1 = \begin{pmatrix}
	I-\Delta  & 0 & T_{1}^{*}\\
	0 & I-\Delta  & -T_{2}^{*}\\
	T_{1} & -T_{2} & -\epsilon^{-1}I_{}
	\end{pmatrix}, 
	\end{equation}
	where $T_i$, $T_i v=v|_{\Gamma}$ for $v\in C(\bar{\Omega}_i)$, $i=1, 2$
	are the trace operators on $H^1(\Omega_1)$ and $H^1(\Omega_2)$, respectively.
	% Hdiv
	Tveito et al. \cite{tveito2017cell} further discuss a mixed formulation of the
	system \eqref{eq:emi_system} where additional unknowns $\sigma_i=-\nabla u_i$, $i=1,2$
	are introduced. If homogeneous Dirichlet boundary conditions $u_1=0$ on $\partial\Omega$
	are assumed the mixed formulation leads to operator
	$\mathcal{A}_2: W_2\rightarrow W'_2$
	\begin{equation}\label{eq:hdivA}
	\mathcal{A}_2 = \begin{pmatrix}
	I & \nabla & T^{*}\\
	-\nabla\cdot& -I & 0\\
	T & 0 & -\epsilon I
	\end{pmatrix},
	\end{equation}
	with $W_2=H(\text{div}, \Omega)\times L^2(\Omega) \times (H^{1/2}(\Gamma)\bigcap \epsilon^{1/2} L^2(\Gamma))$
	and $T$ the normal trace operator $T\sigma= \sigma|_{\Gamma}\cdot n$ for $v\in \left[C(\Omega)\right]^2$. We remark 
	that operators $\mathcal{A}_1$ and $\mathcal{A}_2$ also arise naturally in the analysis of non-overlapping domain decomposition methods 
	for second order elliptic problems in the primal \cite{wohlmuth2000mortar} and mixed formulation \cite{cowsar1995balancing} 
	respectively.
	
	% Isomorh and operator preconditioning
	Assuming that the operators $\mathcal{A}_1$ and $\mathcal{A}_2$ are isomorphisms
	on their respective spaces\footnote{
		The proof of this result as well as stable finite element discretization of the
		problem are subject of current work and will be reported elsewhere.
		We remark that operator $\mathcal{A}_1$ in the limit case $\epsilon=\infty$ has
		been studied in \cite{lamichhane2004mortar} in the context of mortar finite element method.
	}  
	the preconditioners can be established within the framework of operator
	preconditioning \cite{mardal2011preconditioning}. In particular, the Riesz
	map preconditioner for \eqref{eq:mortarA} is 
	\begin{equation}\label{eq:mortarB}
	\mathcal{B}_1=\begin{pmatrix}
	I-\Delta & &\\
	& I-\Delta  & \\
	& & \epsilon^{-1}I + (-\Delta + I)^{-1/2}
	\end{pmatrix}^{-1},
	\end{equation}
	while \eqref{eq:hdivA} will be preconditioned by 
	\begin{equation}\label{eq:hdivB}
	\mathcal{B}_2=\begin{pmatrix}
	I - \nabla \nabla\cdot & &\\
	& I & \\
	& & \epsilon I + (-\Delta + I)^{1/2}
	\end{pmatrix}^{-1}.
	\end{equation}
	Note that the operator sums in $\mathcal{B}_1$, $\mathcal{B}_2$ are due to
	the fact that the interface spaces are intersection spaces \cite{bergh}.
	
	In order to simplify the setting and focus only on the fractional operators
	in the preconditioners we remove the parameter dependence from the problems by
	settting $\epsilon=\infty$ in \eqref{eq:mortarA}, \eqref{eq:mortarB} and similarly
	$\epsilon=0$ for \eqref{eq:hdivA}, \eqref{eq:hdivB}. In turn, the interface spaces
	reduce to $H^{-1/2}(\Gamma)$ and $H^{1/2}(\Gamma)$ respectively and the multilevel algorithm
	is directly applicable for the related interface problems which now involve
	the operator $I-\Delta$, cf. the Laplacian operator in the previous sections. 
	
	% Discretization FEM
	Robustness of $\mathcal{B}_1$, $\mathcal{B}_2$, and in particular
	the fractional Sobolev space preconditioner, are finally demonstrated by
	observing the iteration counts of the preconditioned MinRes method.
	In the experiments we let $\Omega=\left[0, 1\right]^2$ and $\Omega_2=\left[0.25, 0.75\right]^2$.
	The finite element discretization of $W_1$ uses continuous linear Lagrange elements.
	For $W_2$ the first subspace is constructed from the $H(\text{div})$-conforming 
	lowest order Raviart-Thomas elements and the remaining subspaces use piecewise 
	constant elements. The discrete preconditioners shall use off-the-shelve methods for the first two
	blocks. More specifically, a single $V$ cycle of algebraic multigrid 
	is used for $\mathcal{B}_1$ while for $\mathcal{B}_2$ the action is computed
	exactly by a direct solver. The final block of the preconditioners is realized by the 
	proposed multilevel preconditioner with different number of levels $J=2, 3, 4$. 
	We note that with the choice of discretization for the space $W_2$ the fractional 
	multigrid algorithm is applied outside of the setting used in the 
	analysis of Section \ref{sec:discrete-bpx}.
	
	The number of MinRes iterations is shown in Table \ref{tab:hdiv} and
	Table \ref{tab:mortar}. Here, the iterations were started from a random initial
	vector and terminated once the relative preconditioned residual norm was less then
	$10^{-8}$ in magnitude. For both $\mathcal{B}_1$ and $\mathcal{B}_2$ the iterations
	are bounded in the discretization parameter. The linear dependence on the
	number of levels is nicely visible in the results of $\mathcal{B}_2$.
	% Comparison
	The tables further list iteration counts for preconditioners where the fractional operators
	were realized in terms of spectral decomposition. As expected from the theory
	and experiments for the Laplace problem the difference in iteration counts
	between the multilevel realization and specral realization is larger for
	$\mathcal{B}_1$ then it is for $\mathcal{B}_2$.

	% mg = 4
	%# , mes: 1, randomic: 1, B: mg, log: v0.1-6-g5c862e2/paper_hdiv_1e-15_4, minres: petsc, demo: paper_hdiv.py, relconv: 1, eps_param: 1e-15, n: 7, Q: iters, eta: 1.0, tol: 1e-08, error: 1, nlevels: 4, D: 2, s: 0.5

	% mg = 3
	%# , mes: 1, randomic: 1, B: mg, log: v0.1-6-g5c862e2/paper_hdiv_1e-15_3, minres: petsc, demo: paper_hdiv.py, relconv: 1, eps_param: 1e-15, n: 8, Q: iters, eta: 1.0, tol: 1e-08, error: 1, nlevels: 3, D: 2, s: 0.5

	% mg = 2
	%# , mes: 1, randomic: 1, B: mg, log: v0.1-6-g5c862e2/paper_hdiv_1e-15_2, minres: petsc, demo: paper_hdiv.py, relconv: 1, eps_param: 1e-15, n: 9, Q: iters, eta: 1.0, tol: 1e-08, error: 1, nlevels: 2, D: 2, s: 0.5

	% mg = 0
	%# , mes: 1, randomic: 1, B: eig, log: v0.1-6-g5c862e2/paper_hdiv_1e-15_0, minres: petsc, demo: paper_hdiv.py, relconv: 1, eps_param: 1e-15, n: 9, Q: iters, eta: 1.0, tol: 1e-08, error: 1, nlevels: 1, D: 2, s: 0.5

	\begin{table}
		\begin{tabular}{c|c|c||ccc|c}
			\hline
			\multirow{2}{*}{$h$} & \multirow{2}{*}{$\dim W_h$} & \multirow{2}{*}{\#cells$_{\Gamma}$} & \multicolumn{3}{c|}{\#MG} & \multirow{2}{*}{\#Eig}\\
			\cline{4-6}
			& & & $J=2$ & $J=3$ & $J=4$ & \\
			\hline
			%3.54E-01 & 96 & 8 & 12 & -- & -- & 12\\
			%1.77E-01 & 352 & 16 & 16 & 16 & -- & 15\\
			%8.84E-02 & 1344 & 32 & 19 & 19 & 19 & 16\\
			%4.42E-02 & 5248 & 64 & 22 & 23 & 24 & 21\\
			2.21E-02 & 20736 & 128 & 27 & 28 & 27 & 22\\
			1.10E-02 & 82432 & 256 & 27 & 32 & 32 & 22\\
			5.52E-03 & 328704 & 512 & 27 & 33 & 36 & 22\\
			2.76E-03 & 1312768 & 1024 & 27 & 33 & 40 & 22\\
			1.38E-03 & 5246976 & 2048 & 25 & 35 & 40 & 22\\
			\hline
		\end{tabular}
		\caption{Number of MinRes itarations for the operator $B_2\mathcal{A}_2$
			and $\epsilon=10^{-15}$ using multilevel algorithm with $J$ levels as a preconditioner for
			$(-\Delta + I)^{1/2}$. Realizing the fractional operator with spectral
			decomposition leads to iteration counts in the last column.
		}
		\label{tab:hdiv}
	\end{table}

% mg = 4
%# , mes: 1, randomic: 1, B: mg, log: v0.1-6-g5c862e2/paper_mortar_1e+15_4, minres: petsc, demo: paper_mortar.py, relconv: 1, eps_param: 1e+15, n: 7, Q: iters, eta: 1.0, tol: 1e-08, error: 1, nlevels: 4, D: 2, s: -0.5

% mg = 3
%# , mes: 1, randomic: 1, B: mg, log: v0.1-6-g5c862e2/paper_mortar_1e+15_3, minres: petsc, demo: paper_mortar.py, relconv: 1, eps_param: 1e+15, n: 8, Q: iters, eta: 1.0, tol: 1e-08, error: 1, nlevels: 3, D: 2, s: -0.5

% mg = 2
%# , mes: 1, randomic: 1, B: mg, log: v0.1-6-g5c862e2/paper_mortar_1e+15_2, minres: petsc, demo: paper_mortar.py, relconv: 1, eps_param: 1e+15, n: 9, Q: iters, eta: 1.0, tol: 1e-08, error: 1, nlevels: 2, D: 2, s: -0.5

% mg = 0
%# , mes: 1, randomic: 1, B: eig, log: v0.1-6-g5c862e2/paper_mortar_1e+15_0, minres: petsc, demo: paper_mortar.py, relconv: 1, eps_param: 1e+15, n: 9, Q: iters, eta: 1.0, tol: 1e-08, error: 1, nlevels: 1, D: 2, s: -0.5

\begin{table}
	\begin{tabular}{c|c|c||ccc|c}
		\hline
		\multirow{2}{*}{$h$} & \multirow{2}{*}{$\dim W_h$} & \multirow{2}{*}{\#cells$_{\Gamma}$} & \multicolumn{3}{c|}{\#MG} & \multirow{2}{*}{\#Eig}\\
		\cline{4-6}
		& & & $J=2$ & $J=3$ & $J=4$ & \\
		\hline
		%3.54E-01 & 41 & 8 & 21 & -- & -- & 15\\
		%1.77E-01 & 113 & 16 & 33 & 36 & -- & 25\\
		%8.84E-02 & 353 & 32 & 52 & 62 & 72 & 34\\
		%4.42E-02 & 1217 & 64 & 60 & 79 & 89 & 37\\
		2.21E-02 & 4481 & 128 & 67 & 93 & 103 & 36\\
		1.10E-02 & 17153 & 256 & 68 & 92 & 111 & 35\\
		5.52E-03 & 67073 & 512 & 66 & 90 & 112 & 35\\
		2.76E-03 & 265217 & 1024 & 64 & 90 & 112 & 34\\
		1.38E-03 & 1054721 & 2048 & 64 & 88 & 108 & 33\\
		\hline
	\end{tabular}
	\caption{Number of MinRes itarations for the operator $B_1\mathcal{A}_1$
		and $\epsilon=10^{15}$ using multilevel algorithm with $J$ levels as a preconditioner for
		$(-\Delta + I)^{-1/2}$. Realizing the fractional operator with spectral
		decomposition leads to iteration counts in the last column.
	}
	\label{tab:mortar}
\end{table}

%%%%%%%%%%%%%%%%%%%%%%%%%%%%%%%%%%%%%%%%%%%%%%%%%%%%%%%%
%			BIBLIOGRAPHY
%%%%%%%%%%%%%%%%%%%%%%%%%%%%%%%%%%%%%%%%%%%%%%%%%%%%%%%%
	\newpage
	\clearpage
	\bibliographystyle{abbrv}
	\vspace{.125in}
	\bibliography{fractional_mg}
	
\end{document}